\documentclass[12pt]{amsart}
\usepackage{lipsum}

\usepackage[T1]{fontenc}
\usepackage[margin=0.8in]{geometry}
\usepackage[scr=rsfs]{mathalpha}
\usepackage{xspace}
\usepackage{amsfonts}
\usepackage{amsthm}
\usepackage{amssymb}
\usepackage{times} 
\usepackage{graphicx}
\usepackage{hyperref}
\usepackage{tikz}
\usepackage{comment}
\usetikzlibrary{arrows}
\usepackage[all]{xy}
\usepackage{tikz-cd}
\usepackage{enumerate}
\usepackage{mathbbol}
\usepackage{verbatim}
\usepackage{mathtools}

\date{}
\usepackage{mathtools}
\DeclarePairedDelimiter\ceil{\lceil}{\rceil}

\numberwithin{equation}{section}
\newtheorem{thm}{Theorem}[section]
\newtheorem{prop}[thm]{Proposition}
\newtheorem*{thm*}{Main Theorem}
\newtheorem*{lemma*}{Key Lemma}
\tikzset{node distance=5cm, auto}

\newtheorem{lemma}[thm]{Lemma}
\newtheorem{conjecture}[thm]{Conjecture}

\newtheorem{defn}[thm]{Definition}
\newtheorem{cor}[thm]{Corollary}
\newtheorem{remark}[thm]{Remark}

\newtheorem{question}[thm]{Question}
\makeatletter
\newcommand{\etale}{\'etal\@ifstar{\'e}{e\xspace}}
\makeatother
\usepackage[OT2,T1]{fontenc}
\DeclareSymbolFont{cyrletters}{OT2}{wncyr}{m}{n}
\DeclareMathSymbol{\Sha}{\mathalpha}{cyrletters}{"58}
\usepackage{amssymb}

\begin{document}

\title{On  Milnor $K$-theory in the imperfect residue case and  applications to period-index problems}
\author{Srinivasan Srimathy}
\address{School of Mathematics, Tata Institute of Fundamental Research, Mumbai, India}
\email{srimathy@math.tifr.res.in}

\date{}

\begin{abstract}
    Given a $(0,p)$-mixed characteristic complete discrete valued field $\mathcal{K}$  we define a class of finite field extensions called \emph{pseudo-perfect} extensions such that the natural restriction map on the mod-$p$ Milnor $K$-groups is trivial for all $p\neq 2$. This implies that  pseudo-perfect extensions split every element in $H^i(\mathcal{K},\mu_p^{\otimes i-1})$ yielding period-index bounds for Brauer classes as well as higher cohomology classes of $\mathcal{K}$. As a corollary, we prove a conjecture of Bhaskhar-Haase that the Brauer $p$-dimension of $\mathcal{K}$ is upper bounded by $n+1$ where $n$ is the $p$-rank of the residue field. When $\mathcal{K}$ is the fraction field of a complete regular ring,   we show that  any $p$-torsion element in $Br(\mathcal{K})$ that is nicely ramified is split by a pseudo-perfect extension  yielding a bound on its index. We then use patching techniques of Harbater, Hartmann and Krashen to show that the Brauer $p$-dimension of semi-global fields of residual characteristic $p$ is at most $n+2$ and also  give  uniform $p$-bounds for higher cohomologies. These  bounds are sharper than previously  known in \cite{ps14} and \cite{ps_uniform}.
\end{abstract}
\maketitle

\section{Introduction}
Classically, any element $\alpha$ in the Brauer group  $Br(F)$ of a field $F$ is associated with two main invariants namely, its period denoted by $per(\alpha)$, which is its order in the Brauer group and its index, denoted by $ind(\alpha)$, which is the g.c.d of the degrees of finite dimensional splitting fields of $\alpha$ over $F$. It is well known that $per(\alpha) \mid ind(\alpha)$ and  that they have the same prime factors. The standard period-index problem is to determine the minimum integer $\ell$ (if it exists) such that $ind(\alpha) \mid per(\alpha)^{\ell}$ \emph{for every} $\alpha \in Br(F)$. It is expected that  such an integer should exist for a large class of fields and that it should behave nicely  with respect to field extensions.  The \emph{Brauer dimension} of the field $F$, denoted by $Br~dim(F)$, is an invariant that gives information about this integer. More precisely, the $Br~dim(F)$ is  the least integer $\ell$, if it exists,  such that $ind(\alpha) \mid per(\alpha)^{\ell}$ for every $\alpha \in Br(E)$ and every finite extension $E/F$. If such an $\ell$ does not exist, we set $Br~dim(F) = \infty$. Similarly, one defines the \emph{Brauer $p$-dimension}, denoted by $Br_pdim(F)$ to be the least integer $\ell$,  such that $ind(\alpha) \mid per(\alpha)^{\ell}$ for every $\alpha \in Br(E)[p^\infty]$ and every finite extension $E/F$.\\
\indent The invariant $Br_pdim(F)$ is  known for some important class of fields such as  complete discrete valued fields with residual characteristic coprime to $p$ (\cite{hhk09}) and function fields of curves over such field i.e., semi-global fields  (\cite{saltman97}, \cite{hhk09} and \cite{lie11}).   Later bounds on $Br_pdim(F)$ for $F$ as above were also discovered  when the residual characteristic is   $p$ by Parimala, Suresh,  Bhaskhar and Haase (\cite{ps14}, \cite{nivedita}) and it is shown that these bounds depend on the $p$-rank of the residue field.  Recall that for a field $E$ of characteristic $p$, we have $[E:E^p]=p^{\mathscr{R}_p(E)}$, where $\mathscr{R}_p(E)$ is the $p$-rank of $E$.  When $\mathcal{K}$ is a complete discrete valued field of characteristic zero with residual characteristic $p$,  an optimal  bound for $Br_p dim(\mathcal{K})$ is conjectured in \cite{nivedita}):
\begin{conjecture} (\cite[Conjecture 1.5]{nivedita})\label{conj:main}
Let $\kappa$ denote the the residue field of $\mathcal{K}$. Then 
    \begin{align*}
    \mathscr{R}_p(\kappa) \leq Br_pdim(\mathcal{K}) \leq \mathscr{R}_p(\kappa)+1
    \end{align*}
\end{conjecture}
While the lower bound of the Conjecture is  shown in \cite{chip}, the upper bound is open. We emphasize  that the above  bounds for $Br_pdim(\mathcal{K})$
 are optimal (\cite[\S5]{nivedita}, \cite{chip}).\\
\indent Analogous to the period-index problem for the Brauer groups, one can look at the corresponding problem for higher cohomology groups. Given a field $F$ and a prime $p$, the authors in \cite{hhk_israel} define the \emph{stable $i$- splitting dimension of $F$ at $p$}, denoted by $ssd^i_p(F)$,  to be the least integer $\ell$,  such that $ind(\alpha) \mid p^{\ell}$ for every $\alpha \in H^i(E, \mu_{p}^{\otimes i-1})$ and every finite extension $E/F$. For a given $i,j$ and a set of cohomology classes $S\subseteq  H^i(F, \mu_{p}^{\otimes j})$, one can define the  \emph{index of $S$} to be the g.c.d of the degrees of finite dimensional field extensions of $F$ that split every element in $S$ . Then the \emph{generalized stable $i$-splitting dimension at $p$} of $F$, denoted by $gssd^i_p(F)$,  is defined in \cite{hhk_israel} to be the least integer $\ell$ (if it exists),  such that $ind(S) \mid p^{\ell}$ for every finite subset $S \subseteq  H^i(E, \mu_{p}^{\otimes i-1})$ and every finite extension $E/F$. If such an $\ell$ does not exist,  we set $gssd^i_p(F) = \infty$.   Note that when $i=2$, the finiteness of $gssd^2_p(F)$ is equivalent to saying that  $Br(F)$ is \emph{uniformly $p$-bounded} in the language of \cite{ps_uniform}. These invariants are computed in \cite{hhk_israel} for complete  discrete valued fields and semi-global fields  under the assumption that the associated residual characteristic is coprime to $p$.  \\
\indent  Let $\mathcal{K}$ be the fraction field of a complete regular local ring $\mathcal{R}$ of characteristic zero with residue field $\kappa$ of characteristic $p$.  In this paper, we define a new invariant of $\mathcal{K}$ called the \emph{pseudo-rank} of $\mathcal{K}$, denoted by $\mathscr{R}_{ps}(\mathcal{K})$. It is defined as
\begin{align*}
  \mathscr{R}_{ps}(\mathcal{\mathcal{K}}):=   \mathscr{R}_p(\kappa) + dim~\mathcal{R} 
\end{align*}
Note that when $\mathcal{K}$ is also of characteristic $p$ (i.e., in the $(p,p)$-equicharacteristic case),  $\mathscr{R}_{ps}(\mathcal{\mathcal{K}})$ coincides with $ \mathscr{R}_{p}(\mathcal{\mathcal{K}})$ by Cohen's structure theorem, so one can view the notion of pseudo-rank in the mixed characteristic situations as the analog of $p$-rank. We show that many of the period-index bounds in the mixed characteristic situations depend on this invariant.  In particular, we show that  for $p\neq 2$,
\begin{align*}
Br_pdim(\mathcal{K}) \leq \mathscr{R}_{ps}(\mathcal{K})
\end{align*}
thereby proving  Conjecture \ref{conj:main}  (see Corollary \ref{cor:main}). We also obtain an upper bound for the generalized stable $i$-splitting dimension of the complete discrete valued field $\mathcal{K}$ (Corollary \ref{cor:mainh}):
\begin{align*}
        gssd^i_p(\mathcal{K}) \leq \mathscr{R}_{ps}(\mathcal{K})
    \end{align*}
complementing the results in \cite{hhk_israel},  where the coprime-to-$p$ characteristic case is assumed. As a generalization,    we  derive period-index bounds for Brauer classes of fraction field $\mathcal{K}$ of complete regular local rings $\mathcal{R}$ with residue $\kappa$ of characteristic $p\neq 2$: For $\alpha \in Br(\mathcal{K})[p]$ that is "ramified nicely" (see Corollary \ref{cor:regular} for a precise statement)
\begin{align*}
    ind(\alpha) \mid p^{\mathscr{R}_{ps}(\mathcal{K})}
\end{align*}

The above bounds together with the patching techniques of Harbater, Hartman and Krashen gives a sharper bound  compared to \cite[Theorem 3]{ps14} on the Brauer $p$-dimension of semi-global fields $F = \mathcal{K}(C)$,  where $\mathcal{K}$ is a complete discrete valued field with residue $\kappa$ of characteristic $p\neq 2$
(Theorem \ref{thm:semiglobal}):
\begin{align*}
    Br_pdim(F)  \leq \mathscr{R}_{ps}(\mathcal{K})+1
\end{align*}
Observe that this bound  agrees with the general consensus that the  upper bound for $Br_pdim$ should increase at most by one for fields extending  by transcendence degree one. We remark that all of the bounds mentioned above agree with the analogous  bounds in the $(p,p)$-equicharacteristic case. This follows from Cohen's structure theorem and from a theorem of Albert  that   the Brauer $p$-dimension of a field of characteristic $p$ is upper bounded by its $p$-rank (\cite[Corollary 1.2]{ps14}).\\
\indent We also prove that the Brauer group of  $F$ is uniformly $p$-bounded (Theorem \ref{thm:uniform}) and that
\begin{align*}
    gssd_p^2(F)\leq 2 (\mathscr{R}_{ps}(\mathcal{K})+1)
\end{align*}
 Recall from \cite{ps_uniform} that for an integer $i$ and a prime $p$ not equal to the characteristic of $L$, the field $L$ is said to be \emph{uniformly $(i,p)$-bounded} if there exists an $N$ such that for any finite extension $E/L$ and  any finite subset $S \subseteq H^i(E, \mu_p^{\otimes i})$, $ind(S) \leq N$. Such an $N$ is called an \emph{$(i,p)$-uniform bound of $L$}.  On the  higher cohomological side,  we show that $F$ is uniformly $(i,p)$-bounded for every $i\geq 2$ and that $[F(\zeta):F]p^{2 (\mathscr{R}_{ps}(\mathcal{K})+1)}$ is an $(i,p)$-uniform bound for $F$ where $\zeta$ is a primitive $p$-th root of unity (Corollary \ref{cor:uniform}).  \\
\indent An important feature of the  above results is that not only do we provide the numerical period-index bounds for the cohomology classes but in fact  produce an explicit collection of field extensions which we call as \emph{pseudo-perfect extensions} that attain the bound (\S \ref{sec:pseudo}).  In fact, when $p\neq 2$ these extensions split \emph{all} the elements in $H^i(\mathcal{K}, \mu_{p}^{\otimes i-1})$ (not just any finite subset of it). In other words,  we show that  for any pseudo-perfect extension $\mathcal{L}/\mathcal{K}$, the natural restriction map
\begin{align}\label{eqn:nres}
Res_{\mathcal{L}/\mathcal{K}}^i: H^i(\mathcal{K}, \mu_p^{\otimes i-1}) \rightarrow H^i(\mathcal{L}, \mu_p^{\otimes i-1})
\end{align}
is zero for all $i\geq 2$ (Corollary \ref{cor:mainh}). \\
\indent Loosely speaking, these extensions  mimic stages in the construction of the perfect hull of  the fraction fields of $(p,p)$-equicharacteristic complete  regular local rings. More precisely, these extensions are obtained by adjoining $p$-th roots of a lift of some $p$-basis of the residue and the $p$-th roots  of some regular set of parameters. These  exhibit phenomena similar to the corresponding extensions in the $(p,p)$-equicharacteristic case with respect to the natural restriction maps on cohomologies (\S \ref{sec:pseudo}).\\
\indent At this point, we would like to remark that  the prime $p=2$ behaves differently with respect to the restriction map in $(\ref{eqn:nres})$. In particular, when  $p=2$, the map $Res_{\mathcal{L}/\mathcal{K}}^i$ is still zero for $i\geq 3$ but not necessarily zero for $i=2$. We show that the obstruction  to vanishing when $p=2$  and $i=2$ lies in the  group $H^2_{\acute{e}t}(\kappa,\mathbb{Z}/2(1)) \cong Br(\kappa)[2]$ where $\kappa$ is the residue of $\mathcal{K}$ (Corollary \ref{cor:mainh}), and in fact demonstrate with an example that if this obstruction does not vanish, then $Res_{\mathcal{L}/\mathcal{K}}^2$  is not necessarily zero (\S \ref{sec:char2}).\\
\indent  We heavily use Milnor $K$-theory and Kato's work on cohomologies  of complete discrete valued fields to prove our results. The proofs  use the unit group filtration on mod-$p$ Milnor $K$-groups constructed by Bloch and Kato (\cite{blochkato}),  together with  symbol manipulation techniques to achieve the Main Theorem of the paper (Theorem \ref{thm:main}):
\begin{thm*} Let $\mathcal{K}$ be a complete discrete valued field  with residue $\kappa$ of characteristic $p \neq 2$. For every pseudo-perfect extension $\mathcal{L}/\mathcal{K}$, the natural restriction map on the mod-$p$ Milnor $K$-groups
\begin{align*}
Res_{\mathcal{L}/\mathcal{K}}^i: K_i(\mathcal{K})/p \rightarrow K_i(\mathcal{L})/p
\end{align*}
is zero for all $i\geq 2$.
\end{thm*}
  The following  key lemma (Lemma \ref{lemma:main2})  illustrates the  vanishing  property of symbols in mod-$p$ Milnor $K$-groups whose entries are sum of $p^{th}$-powers. It plays pivotal role in proving the Main Theorem.
\begin{lemma*}
    Let $i\geq 2$. Suppose $\mathcal{K}$ is  complete  discrete valued field with residue  field  of characteristic $p \neq 2$ and let $\mathfrak{w} =\{w_1,w_2, \cdots, w_i\} \in K_i(\mathcal{K})/p$ be such that each $w_i \in \mathcal{O}_{\mathcal{K}}^{\times}$ and is a finite sum of $p$-{th} powers in $\mathcal{O}_{\mathcal{K}}$. Then $\mathfrak{w} = 0$.
    \end{lemma*}
See Theorem \ref{thm:main} and Lemma \ref{lemma:main2} for  precise statements that consider the case of $p=2$ as well.\\
\indent This paper is organized as follows: We define pseudo-perfect extensions and study their properties in \S \ref{sec:pseudo}.  We use this notion, to state the main results of the paper in \S\ref{sec:mainresults}.   We recall the necessary material in the literature that is used in the paper in \S\ref{sec:prelim}. The proof of the main theorem is shown in \S\ref{sec:mainthm}. The anomaly at $p=2$ is demonstrated in \S\ref{sec:char2} with an explicit example. Finally, in \S\ref{sec:applications}, we apply our results to show many period-index bounds for Brauer groups as well as higher cohomologies.

\section{Notations}
 For a ring $R$, $dim~R$ denotes the dimension of $R$. The  Brauer group of a field $F$ is denoted by $Br(F)$ and its  
 $p$-torsion part is denoted by  $Br(F)[p]$. For a central simple algebra $A$ over $F$, the symbol $[A]$ denotes its class in $Br(F)$. The $n$-th Milnor $K$-group of $F$ is denoted by $K_n(F)$. The symbols in $K_n(F)$ and  $K_n(F)/p$ are denoted by $\{a_1, a_2, \cdots, a_n\}$. All the residue fields of  local rings in this paper will be of characteristic $p>0$ and of  finite $p$-rank. For a discrete valued field $\mathcal{K}$, the ring of integers and its  maximal ideal are denoted by $\mathcal{O}_{\mathcal{K}}$ and $\mathfrak{m}_{\mathcal{K}}$ respectively. The residue of an  element $a\in \mathcal{O}_{\mathcal{K}}$  is denoted by $\overline{a}$.  The finite field with $p$ elements is denoted by $\mathbb{F}_p$. 

 \section{Pseudo-Perfect extensions}\label{sec:pseudo}
 For a set $S$ and an integer $\ell$, let $S^{1/p^{\ell}}$ denote the set $\{s^{1/p^{\ell}}| s \in S\}$. Let $\kappa$ be a field of characteristic $p >0$. Assume that the $p$-rank of $\kappa$ is finite. Let $\mathcal{B} = \{b_1, b_2, \cdots, b_n\}$ be a $p$-basis for $\kappa$. Let  $\kappa^{1/p^{\ell}}$ denote the field extension of $\kappa$ obtained by adjoining $1/p^{\ell}$-th roots of elements in $\kappa$.  
 
 Observe that  
 \begin{align*}
 \kappa^{1/p^{\ell}} = \kappa(\mathcal{B}^{1/p^{\ell}})
 \end{align*}

\subsection{The equi-characteristic case}
Let $\mathcal{K}$  be the fraction field of a complete regular local ring $\mathcal{R}$ with residue $\kappa$ of characteristic $p$.  Suppose  $\mathcal{K}$ is also of characteristic $p$, then by Cohen's structure theorem $\mathcal{K} \simeq \kappa((\pi_1, \pi_2, \cdots, \pi_n))$ for some $\pi_1, \pi_2, \cdots \pi_n \in \mathcal{K}$.  Let $\Pi = \{\pi_1, \pi_2, \cdots \pi_n\}$. Recall that if $\mathcal{B}$ is a $p$-basis for $\kappa$, then $\Lambda =\mathcal{B} \cup \Pi$ is a $p$-basis for $\mathcal{K}$ and therefore
 \begin{align*}
     \mathcal{K}^{1/p^{\ell}} = \mathcal{K}( \Lambda^{1/p^{\ell}})
 \end{align*}
Since $\mathcal{K}^{1/p^{\infty}}$ is the perfect closure of $\mathcal{K}$, one may think of $\mathcal{K}( \Lambda^{1/p^{\ell}})$ as the $\ell$-th level while constructing the perfect closure of $\mathcal{K}$.\\
\indent Let  $K_i(\mathcal{K})$ denote the $i$-th Milnor $K$-group of $\mathcal{K}$. Suppose $\mathcal{L} = \mathcal{K}( \Lambda^{1/p}) = \mathcal{K}^{1/p}$.  Then note that  $\mathcal{K} \subseteq \mathcal{L}^p$ and hence  the natural restriction map
    \begin{align}\label{eqn:restriction}
Res_{\mathcal{L}/\mathcal{K}}^i: K_i(\mathcal{K})/p\rightarrow K_i(\mathcal{L})/p
\end{align}
is zero for all $i\geq 1$.
\subsection{Mixed characteristic case}
Now, suppose $\mathcal{K}$ is of characteristic zero so that we are in the mixed characteristic case and as before let $\mathcal{B}$ be a $p$-basis for $\kappa$ and let $\Pi$ denote a regular system of parameters of $\mathcal{R}$. 
 \begin{defn}\normalfont
     Let $\mathcal{K}$ be the fraction field of a complete regular local ring $\mathcal{R}$ of characteristic zero with residue field $\kappa$ of characteristic $p$.  Let $\Pi$ denote a regular system of parameters of $\mathcal{R}$. Let $\mathcal{B}$ denote a $p$-basis of  $\kappa$ and let $\tilde{\mathcal{B}}$ denote a lift of $\mathcal{B}$ in $\mathcal{R}$. Then the set $\Lambda =\tilde{\mathcal{B}}\cup \Pi$ is called a \emph{pseudo-basis} of $\mathcal{K}$ and its cardinality is a well-defined invariant of $\mathcal{K}$ called the \emph{pseudo-rank} of $\mathcal{K}$ denoted by $\mathscr{R}_{ps}(\mathcal{K})$. 
 \end{defn}
Note that 
\begin{align*}
  \mathscr{R}_{ps}(\mathcal{\mathcal{K}})=   \mathscr{R}_p(\kappa) + dim~\mathcal{R} 
\end{align*}

 \begin{defn}\normalfont
     A field extension $\mathcal{L}/\mathcal{K}$ is called  \emph{pseudo-perfect extension of level $\ell$}  over $\mathcal{K}$ if 
     \begin{align}
         \mathcal{L} \simeq \mathcal{K}(\Lambda^{1/p^{\ell}})
     \end{align}
     for some pseudo-basis $\Lambda$ of $\mathcal{K}$. Let $\mathfrak{PP}(\mathcal{K}, \ell)$ denote the collection of all pseudo-perfect extensions of level $\ell$ over $\mathcal{K}$. 
    \\
    \indent The following lemma demonstrates that pseudo-rank is invariant under   pseudo-perfect extensions.
     \begin{lemma}
         With notations as above, any  $ \mathcal{L} \simeq \mathcal{K}(\Lambda^{1/p^{\ell}})$ is the fraction field of a complete regular local ring. Moreover, if $\Lambda$ is a pseudo-basis for $\mathcal{K}$, then $\Lambda^{1/p^{\ell}}$ is a pseudo-basis for $\mathcal{L}$. In particular, $\mathscr{R}_{ps}(\mathcal{L}) =\mathscr{R}_{ps}(\mathcal{K})$.
     \end{lemma}
     \begin{proof}
         Let $\mathcal{K}$ be the fraction field of a complete regular local ring $\mathcal{R}$.  Let
         \begin{align*}
             \mathcal{S} = \mathcal{R}[\Lambda^{1/p^{\ell}}]
         \end{align*}
         Then $\mathcal{S}$ is a complete local ring by \cite[Theorem 7]{cohen_structure}. Moreover,  it also regular by \cite[Corollary 3.3]{ps14}.  Clearly, the fraction field of $\mathcal{S}$ is $\mathcal{L}$. It is clear that if  $\Pi$  is a regular system of parameters for $\mathcal{R}$, then $\Pi^{1/p^{\ell}}$ is a regular system of parameters for $\mathcal{S}$ and that  its residue field is $\kappa^{1/p^{\ell}}$.  The rest of the claim follows form this observation.
     \end{proof}
     
     When $\ell =1$, we simply write $\mathfrak{PP}(\mathcal{K})$ to denote the collection of pseudo-perfect extensions of level one.
     \end{defn}

 \begin{remark} \label{rmk:degree} 
 We note that for every  $\mathcal{L} \in \mathfrak{PP}(\mathcal{K})$, 
 \begin{align*}
     [\mathcal{L}:\mathcal{K}] = p^{\mathscr{R}_{ps}(\mathcal{\mathcal{K}})}
 \end{align*}
\end{remark}
Let $\mathcal{L} = \mathcal{K}(\Lambda^{1/p})  \in \mathfrak{PP}(\mathcal{K})$. One of the main goals of this paper is to understand the analogous  natural restriction map of (\ref{eqn:restriction}) in the mixed characteristic case. 
 
\section{Main Results}\label{sec:mainresults}
For a field $\mathcal{K}$, let  $K_*(\mathcal{K})$ denote the Milnor $K$-groups of $\mathcal{K}$.

\begin{thm}\label{thm:main} Let $\mathcal{K}$ be a complete  discrete valued field  of characteristic zero and residue field $\kappa$ of characteristic $p$. Then for every $\mathcal{L} \in \mathfrak{PP}(\mathcal{K})$, the natural restriction map
\begin{align*}
Res_{\mathcal{L}/\mathcal{K}}^i: K_i(\mathcal{K})/p \rightarrow K_i(\mathcal{L})/p
\end{align*}
is zero for all $i\geq 3$. Moreover, $Res_{\mathcal{L}/\mathcal{K}}^2$ is   zero  if  either $p\neq 2$ or $H^2_{\acute{e}t}(\kappa,\mathbb{Z}/2(1))=0$. 
\end{thm}
\begin{proof}
    See \S\ref{sec:mainthm}.
    \end{proof}
\begin{remark}\normalfont
    When $p=2$, the condition in  Theorem \ref{thm:main} that $H^2_{\acute{e}t}(\kappa,\mathbb{Z}/2(1)) =0$ is satisfied if $\kappa$ is Artin-Schreier closed (i.e., there are no Artin-Schreier extensions) or perfect. See Remark \ref{rmk:asclosed}.
\end{remark}
\begin{remark}\normalfont
 If $p=2$ and $H^2_{\acute{e}t}(\kappa,\mathbb{Z}/2(1)) \neq 0$, $Res_{\mathcal{L}/\mathcal{K}}^2$  is not necessarily  zero. So the condition that  $H^2_{\acute{e}t}(\kappa,\mathbb{Z}/2(1)) = 0$ is necessary for the assertion to hold if $p=2$.  See  
 Theorem \ref{thm:char2}.
\end{remark}

\begin{cor}\label{cor:mainh}
 With notations as above, Theorem \ref{thm:main} holds if $K_i(\mathcal{K})/p$ is replaced with $H^i(\mathcal{K}, \mu_p^{\otimes i-1})$ i.e., the natural restriction map
\begin{align*}
Res_{\mathcal{L}/\mathcal{K}}^i: H^i(\mathcal{K}, \mu_p^{\otimes i-1}) \rightarrow H^i(\mathcal{L}, \mu_p^{\otimes i-1})
\end{align*}
is zero for all $i\geq 3$. Moreover, $Res_{\mathcal{L}/\mathcal{K}}^2$ is zero  if  either $p\neq 2$ or $H^2_{\acute{e}t}(\kappa,\mathbb{Z}/2(1)) =0$. In particular, 
\begin{align}\label{eqn:gssd}
        gssd^i_p(\mathcal{K}) \leq \mathscr{R}_{ps}(\mathcal{K})
    \end{align}
    holds unconditionally for $i\geq 3$  as well as for  $i=2$ provided  either $p\neq 2$ or $H^2_{\acute{e}t}(\kappa,\mathbb{Z}/2(1)) =0$.
\end{cor}
\begin{proof}
    See \S\ref{sec:mainh} for the first part of the claim. The second part follows by the observation  that the $p$-rank of the residue is invariant under  any  finite extension $\mathcal{K}'$ of $\mathcal{K}$ so that for any $\mathcal{L}' \in \mathfrak{PP}(\mathcal{K}')$, 
     $[\mathcal{L}':\mathcal{K}'] = p^{\mathscr{R}_{ps}(\mathcal{K})}$ (Remark \ref{rmk:degree}).
    \end{proof}
\begin{cor} (\cite[Conjecture 1.5]{nivedita})\label{cor:main}
 With notations as above, 
    \begin{align*}
    \mathscr{R}_p(\kappa) \leq Br_pdim(\mathcal{K}) \leq \mathscr{R}_{ps}(\mathcal{K})
    \end{align*}
  if either $p\neq 2$ or   $H^2_{\acute{e}t}(\kappa,\mathbb{Z}/2(1))$ is trivial.  
\end{cor}
\begin{proof}
    The lower bound is already shown in  \cite{chip}.  As argued before, the $p$-rank of the residue  is invariant under  any  finite extension $\mathcal{K}'$ of $\mathcal{K}$, so for any $\mathcal{L}' \in \mathfrak{PP}(\mathcal{K}')$, 
     $[\mathcal{L}':\mathcal{K}'] = p^{\mathscr{R}_{ps}(\mathcal{K})}$ (Remark \ref{rmk:degree}). The upper bound now follows from the fact that $Br(\mathcal{K})[p] \simeq H^2(\mathcal{K}, \mu_p)$  together with Corollary \ref{cor:mainh} and \cite[Lemma 1.1]{ps14}.
    \end{proof}

\begin{cor}
 With notations as above, let $[D] \in Br(\mathcal{K})[p]$ be the class of a division algebra $D$ of degree $p^{\mathscr{R}_{ps}(\mathcal{K})}$ over $\mathcal{K}$. Then every $\mathcal{L} \in \mathfrak{PP}(\mathcal{K})$ embeds as a maximal subfield of $D$.
\end{cor}
\begin{proof}
    Follows from  Corollary \ref{cor:mainh} and standard theory on central simple algebras.
\end{proof}
The following is a generalization of Corollary \ref{cor:mainh} for higher dimensional regular local rings  when $i=2$.
 
    \begin{thm}\label{thm:ndim}
    Let $\mathcal{K}$ be  the fraction field of a complete  regular local ring $\mathcal{R}$ with   residue field $\kappa$  of characteristic $p \neq 2$. For $\alpha \in Br(\mathcal{K})[p]$, suppose the  ramification locus  of $\alpha$ corresponds to a subset of some regular system of parameters $\Pi$ of $\mathcal{R}$. Then $\alpha$ is split by some $\mathcal{L} \in \mathfrak{PP}(\mathcal{K})$. More specifically, for every pseudo-basis $\Lambda$ of $\mathcal{K}$ containing $\Pi$,
    \begin{align*}
        \alpha \otimes_{\mathcal{K}} \mathcal{K}(\Lambda^{1/p}) =0
    \end{align*}
In other words, the kernel of the natural map 
    \begin{align*}
        Br(\mathcal{K})[p] \rightarrow Br(\mathcal{K}(\Lambda^{1/p}))[p]
    \end{align*}
    contains all the classes that are ramified at most at $\Pi$.
\end{thm}
\begin{proof}
    See \S\ref{sec:higher}.
\end{proof}
\begin{cor}\label{cor:regular}
    With notations  and the hypothesis in Theorem \ref{thm:ndim}, $ind(\alpha) \mid p^{\mathscr{R}_{ps}(\mathcal{K})}$.
\end{cor}

 The above period-index bounds  for elements in $Br(\mathcal{K})[p]$ yield an improved upper bound for semi-global fields  via patching techniques of \cite{hhk09}: 
\begin{thm}\label{thm:semiglobal}
Let $\mathcal{K}$ be a complete discrete valued field  with residue $\kappa$. Assume that either $p=char~\kappa \neq2$ or that $H^2_{\acute{e}t}(\kappa,\mathbb{Z}/2(1))$ is trivial. Let $F$ be  a semi-global field over $\mathcal{K}$ i.e., the function field of a curve over $\mathcal{K}$. Then 
\begin{align*}
    Br_pdim(F) \leq \mathscr{R}_{ps}(\mathcal{K})+1
\end{align*}

\end{thm}

\begin{proof}
    See \S \ref{sec:semiglobal}.
\end{proof}
This bound is sharper than known before (compare \cite[Corollary 3.7]{ps14}) and agrees with the general consensus that the  upper bound for $Br_pdim$ should increase at most by one for fields extending  by transcendence degree one.\\
\indent The above results also yield  uniform bounds for the Brauer group and  higher cohomologies of semi-global fields. These are sharper than previously known in \cite{ps_uniform}:  

\begin{thm}\label{thm:uniform}
With notations in Theorem \ref{thm:semiglobal}, the Brauer group of  $F$ is uniformly $p$-bounded. More precisely, 
\begin{align*}
    gssd_p^2(F)\leq 2 (\mathscr{R}_{ps}(\mathcal{K})+1)
\end{align*}
    In addition, $F$ is uniformly $(2,p)$-bounded and $[F(\zeta):F]p^{2(\mathscr{R}_{ps}(\mathcal{K})+1)}$ is a $(2,p)$-uniform bound of $F$, where $\zeta$ is a primitive $p$-th root of unity.
\end{thm}
\begin{proof}
    See \S \ref{sec:uniform}.
    \end{proof}

    \begin{cor}\label{cor:uniform}
    With notations as above, $F$ is uniformly $(i,p)$-bounded for every $i\geq 2$ and $[F(\zeta):F]p^{2 (\mathcal{R}_{ps}(\mathcal{K})+1)}$ is an $(i,p)$-uniform bound for $F$.
    \end{cor}
\begin{proof}
    This follows from Lemma  \cite[Lemma 1.6]{ps_uniform}.
\end{proof}
\begin{remark}\normalfont
    In all of the above results where patching is not used, one may feel free to replace "complete" with "Henselian". Patching techniques for semi-global fields $\mathcal{K}(C)$ require that $\mathcal{K}$ is complete.
    \end{remark}
\section{Preliminaries}\label{sec:prelim}
\subsection{Cohomologies in the positive characteristic} For a scheme $X$ and  integers $r,n \in \mathbb{Z}$ with $n\neq 0$, the object $(\mathbb{Z}/n)(r)$ is defined as  follows. If $n$ is invertible in $X$, then $(\mathbb{Z}/n)(r)$ is the $r^{th}$ Tate twist of the constant sheaf on $X_{\acute{e}t}$. If $X$ is smooth over a field of characteristic $p$ and $n=p^sm, p\nmid m$, $(\mathbb{Z}/n)(r)$ is an object  of  the the derived category of abelian sheaves on $X_{\acute{e}t}$ defined by
\begin{align*}
    (\mathbb{Z}/n)(r) = (\mathbb{Z}/m)(r) + W_s\Omega^r_{X,log}[-r]
    \end{align*}
where $W_s\Omega_{X,log}$  is the logarithmic part of the de Rham-Witt complex on $X_{\acute{e}t}$.\\
\indent Let $K$ be a field and $n\neq 0$. Following \cite{kato_swan},  we define  
\begin{align*}
    H^q_n(K)&:= H^q((Spec~K)_{\acute{e}t}, (\mathbb{Z}/n)(q-1))\\
    H^q(K)&:= \varinjlim_{n} H^q_n(K)
\end{align*}
We recall that
\begin{align*}
    H^2(K) \cong Br(K)
  \end{align*}
  In particular,
  \begin{align*}
    H^2(K, (\mathbb{Z}/p)(1)) \cong Br(K)[p]
  \end{align*}
  \begin{remark}\label{rmk:asclosed}\normalfont
      When $char~K = p$,  $Br(K)[p]$ is generated by cyclic $p$-algebras (\cite[Chapter VII, \S9, Theorem 31]{albert_book}). Therefore, if $K$ is either perfect or Artin-Schreier closed, then $H^2_{\acute{e}t}(K, (\mathbb{Z}/p)(1)) =0$.  Here, $K$ is said to be Artin-Schreier closed if there are no non-trivial Artin-Schreier extensions of $K$. This is equivalent to saying that
      \begin{align*}
          K/\mathcal{P}(K) \cong H^1_{\acute{e}t}(K, \mathbb{Z}/p) = 0
      \end{align*}
      where $\mathcal{P}$ is the Artin-Schreier operator on $K$.
  \end{remark}
Let $\Omega^n_{F,log}$ denote the group of logarithmic K\"{a}hler differentials.  For a field $F$ of characteristic $p>0$, let 
\begin{align*}
\nu(q)_F: = Fr-I: \Omega_F^q &\rightarrow \Omega_F^q/d\Omega_F^{q-1}\\
x\cdot d\log y_1\wedge d\log y_2\wedge\cdots  \wedge  d\log y_q &\mapsto (x^p-x)\cdot d\log y_1\wedge d\log y_2\wedge\cdots  \wedge  d\log y_q 
\end{align*}
where $Fr$ is the Frobenius. Then
    \begin{align}\label{eqn:brauerp}
    H^q(F, (\mathbb{Z}/p)(q)) &\cong \Omega^q_{F,log} \cong ker(\nu(q)_F) \nonumber\\
        H^q_p(F) = H^1(F, \Omega_{F,\log}^{q-1})&\cong coker(\nu(q-1)_F)) =\Omega_F^{q-1}/((Fr-I)\Omega_F^{q-1} + d \Omega_F^{q-2})
    \end{align}
\begin{remark}\label{rmk:brauerp}\normalfont
    When $q=2$,  $\Omega_F/((Fr-I)\Omega_F + dF)\cong H^2_p(F) \cong  Br(F)[p] $ by (\ref{eqn:brauerp}) and the isomorphism is given by
    \begin{align*}
        wd \log v \mapsto [w,v)
    \end{align*}
    where the symbol $[w,v)$ denotes the cyclic $p$-algebra 
$$F<x,y : x^p-x=w, y^p=v, yxy^{-1} = x+1>$$
\end{remark}
\subsection{Relation to Milnor $K$-theory}
For a field $F$, let $K_n(F)$ denote the $n$-{th} Milnor $K$-group of $F$. Recall that $K_n(F)$ is the quotient of $(F^{\times})^{\otimes n}$ by the subgroup generated by the  elements of the form $x_1\otimes x_2\otimes\cdots\otimes x_n$ for which $x_i+x_j =1$ for some $1\leq i<j\leq n$.\\
\indent For any natural number  $m$, the symbols in $K_n(F)$  as well as  $K_n(F)/m$ are denoted by $\{a_1, a_2,\cdots,a_n\}, a_i \in F^{\times}$. 
\begin{thm}[Norm residue isomorphism \cite{merk_suslin}, \cite{voe03}, \cite{voe11}]\label{thm:voe}
Let $F$ be a field and let $p$ be an integer invertible in $F$. Then for every $n$, the Galois symbol
\begin{align*}
K_n(F)/p &\rightarrow  H^n(F, (\mathbb{Z}/p)(n))\\
\{x_1, x_2, \cdots, x_n\} &\mapsto  (x_1)\cup (x_2)\cup\cdots\cup (x_n)
\end{align*}
is an isomorphism.
\end{thm}

Similarly by results of Bloch, Kato and Gabber we can identify the mod-$p$ Milnor $K$-theory of a field $F$ of characteristic $p$ with the group of logarithmic K\"ahler differentials.
    \begin{thm}[Bloch-Gabber-Kato theorem \cite{blochkato}]
    Let $F$ be a field of characteristic $p>0$. Then for every $n$, the differential symbol
    \begin{align*}
        K_n(F)/p &\rightarrow  H^n(F, (\mathbb{Z}/p)(n)) \cong \Omega^n_{F,\log}\\
        \{x_1, x_2, \cdots, x_n\} &\mapsto d\log x_1\wedge d\log x_2\cdots \wedge d\log x_n
    \end{align*}
    is an isomorphism.
    \end{thm}

\subsection{Kato's filtration(\cite{kato_swan},\cite{blochkato})}\label{sec:kato}
Let  $\mathcal{K}$  be a complete discrete valued field of characteristic zero with uniformizer $\pi$ and residue $\kappa$ of characteristic $p$.  Let $U_{\mathcal{K}}^i$ denote the $i$-th unit group of $\mathcal{K}$ i.e.,
\begin{align*}
    U_{\mathcal{K}}^i = \{1+ \pi^i \mathcal{O}_{\mathcal{K}}, \times\}
    \end{align*}
For a non-negative integer $i$, let  $U^{i}K_q(\mathcal{K})$ be the subgroup of $K_q(\mathcal{K})$ generated by the symbols of the form 
\begin{align}\label{eqn:unitk}
\{v, y_2, y_3,\cdots, y_q\}, {~~~~~~~~} v \in U_{\mathcal{K}}^i,  y_j \in \mathcal{K}^{\times}
\end{align}
Assume that $\mathcal{K}$ contains a primitive $p$-th root of unity $\zeta$ so that  $H^q(\mathcal{K}, (\mathbb{Z}/p)(q)) \simeq H^q_p(\mathcal{K})$ and we can identify $K_q(\mathcal{K})/p$ with $H^q_p(\mathcal{K})$ via the norm residue isomorphism. For $i\geq 1$, let $fil^i~H^q_p(\mathcal{K})$ denote the subgroup of $H^q_p(\mathcal{K})$ generated by symbols of the form (\ref{eqn:unitk}). Let
\begin{align*}
   e_{\mathcal{K}}'=  e_{\mathcal{K}}p(p-1)^{-1}
\end{align*}
 where $e_{\mathcal{K}}$ is the absolute ramification index of $\mathcal{K}$. Note that by Lemma \ref{lem:unitpth}, $fil^i =0$ for all $i > e_{\mathcal{K}}'$. Let 
\begin{align*}
gr^0 &= H^q_p(\mathcal{K})/fil^0\\
  gr^i &= fil^i/fil^{i+1}, {~~~~~~~} i\geq 1
\end{align*}
Then we have the following isomorphisms
\begin{align*}
    gr^0 &\cong \Omega_{\kappa, \log}^q \oplus  \Omega_{\kappa, \log}^{q-1}\\
    gr^i &= \begin{cases}
        \Omega_{\kappa}^{q-1} \text{~~if~~} p\nmid i\\
        \Omega_{\kappa}^{q-1}/Z_{\kappa}^{q-1}\oplus \Omega_{\kappa}^{q-2}/Z_{\kappa}^{q-2} \text{~~~~else}
    \end{cases}
    \end{align*}
    \begin{align}\label{eqn:tame}
 gr^{e_{\mathcal{K}}'} &= Im(\gamma_{\pi}: H^q_p(\kappa) \oplus H^{q-1}_p(\kappa) \xrightarrow{(\iota_p^q, \iota_p^{q-1})} H^q_p(\mathcal{K}))
\end{align}
where $Z_{\kappa}^{*}$ denote the subgroup of closed forms and $\gamma_{\pi} := (\iota_p^q, \iota_p^{q-1})$  on each of the direct summands is defined as 
\begin{align}\label{eqn:iota1}
    \iota_p^q: H^q_p(\kappa) &\rightarrow H^q_p(\mathcal{K})\\
    \overline{x}d\log \overline{y}_1 \wedge d\log \overline{y}_2\wedge \cdots d\log \overline{y}_{q-1} &\mapsto (1+ (\zeta -1)^px, y_1, y_2,  \cdots, y_{q-1})\nonumber
\end{align}
and
\begin{align}\label{eqn:iota2}
    \iota_p^{q-1}: H^{q-1}_p(\kappa) &\rightarrow H^q_p(\mathcal{K})\\
    \overline{x}d\log \overline{y}_1 \wedge d\log \overline{y}_2\wedge \cdots d\log \overline{y}_{q-2} &\mapsto (1+ (\zeta -1)^px, y_1, y_2,  \cdots, y_{q-2}, \pi)\nonumber
\end{align}
(In the above equations, a letter with a bar on the top denotes the residue of the corresponding letter without bar.)
\begin{remark}\label{rmk:tame}\normalfont
    The map 
    \begin{align*}
         \gamma_{\pi}: H^q_p(\kappa) \oplus H^{q-1}_p(\kappa) \xrightarrow{\cong} gr^{e_{\mathcal{K}}'}
    \end{align*}
    in (\ref{eqn:tame}) is an isomorphism (\cite[Theorem 2]{kato_discrete}). Moreover, there is also an isomorphism
    \begin{align*}
   gr^{e_{\mathcal{K}}'} \cong H^q_{p,tame}(\mathcal{K})
    \end{align*}
    Here, $H^q_{p,tame}(\mathcal{K}) \subseteq H^q_{p}(\mathcal{K})$ is the subgroup of \emph{tamely ramified} elements given by
    \begin{align*}
        H^q_{p,tame}(\mathcal{K}):= ker(H^q_{p}(\mathcal{K}) \rightarrow H^q_{p}(\mathcal{K}_{nr}))
    \end{align*}
    where $\mathcal{K}_{nr}$ is the maximal unramified extension of $\mathcal{K}$.
    \end{remark}
\subsection{Unramified classes and purity for Brauer groups}
For an integral domain $\mathcal{R}$ with field of fractions $\mathcal{K}$, an element $\alpha \in Br(\mathcal{K})$ is said to be \emph{unramified on $\mathcal{R}$} if it is in the image of the natural map
\begin{align*}
    Br(\mathcal{R})\rightarrow Br(\mathcal{K})
    \end{align*}
    or equivalently if there exists an Azumaya algebra $\mathcal{A}$ over $\mathcal{R}$ such that $\alpha = [\mathcal{A}\otimes_{\mathcal{R}} \mathcal{K}]$.\\
    \indent For a regular integral scheme $\mathcal{X}$ with function field $\mathcal{K}$, we say that $\alpha \in Br(\mathcal{K})$ is \emph{unramified at a point $x\in \mathcal{X}$} if $\alpha$ is unramified over the local ring $\mathcal{O}_{\mathcal{X},x}$.  The \emph{ramification divisor} of $\alpha$ is the  divisor given by the sum of all codimension one points $x\in \mathcal{X}$ where $\alpha$ is ramified (i.e., not unramified). The support of the ramification divisor is called  the  \emph{ramification locus of $\alpha$}. \\
    \indent We have the following  as a consequence of the purity result for Brauer groups:
    \begin{thm}(\cite[Theorem 1.2]{purity}) \label{thm:purity}
        For a Noetherian integral regular scheme $\mathcal{X}$ with function field $\mathcal{K}$,
        \begin{align*}
    H^2_{\acute{e}t}(\mathcal{X}, \mathbb{G}_m) = \bigcap_{x\in \mathcal{X}^{(1)}} H^2_{\acute{e}t}(\mathcal{O}_{\mathcal{X},x}, \mathbb{G}_m)
    \end{align*}
    as subgroups of $H^2_{\acute{e}t}(\mathcal{K}, \mathbb{G}_m)$ where $\mathcal{X}^{(1)}$ denotes the codimension one points of $\mathcal{X}$.
    \end{thm}

\section{Proof of Theorem \ref{thm:main}}\label{sec:mainthm}
We will now prove our main Theorem \ref{thm:main}.  We start with a few lemmas.  Throughout this section $\mathcal{K}$ is an complete discrete valued field with residue $\kappa$ of characteristic $p$. Let $\mathfrak{v}$ denote the discrete valuation. As before,
\begin{align*}
    e_{\mathcal{K}}' = e_{\mathcal{K}}p(p-1)^{-1}
\end{align*}
where $e_{\mathcal{K}}$ is the absolute ramification index of  $\mathcal{K}$.
\begin{lemma} \label{lem:unitpth}
 Let $n \in \mathbb{Z}$ be  such that $n > e'_{\mathcal{K}}$. Then
\begin{align*}
    U_{\mathcal{K}}^{n}\in (\mathcal{O}_{\mathcal{K}}^{\times})^p
    \end{align*}
    \end{lemma}
\begin{proof}
    This is well known. See, for example, the proof of \cite[Lemma 2.4 (2)]{kato_swan}.
\end{proof}
 \begin{defn}
   For a  discrete valued field $\mathcal{K}$ and a positive integer $m$, let   $U^{m}K_2(\mathcal{K})$ be the subgroup of $K_i(\mathcal{K})$ generated by the symbols of the form $\{u, y\}$ where $u \in U_{\mathcal{K}}^m$ and $y \in \mathcal{K}^{\times}$ and let $U^{j_1, j_2}K_2(\mathcal{K})$ be  the subgroup of $K_2(\mathcal{K})$ generated by the symbols of the form $\{y_1, y_2\}$ where $y_i \in U_{\mathcal{K}}^{j_i}$.  
 \end{defn}     
 Note that $\{U^{m}K_2(\mathcal{K})\}$ forms a descending filtration on $K_2(\mathcal{K})$. Now we recall  a lemma of Kato which  plays a crucial role in  the proof of the Key Lemma \ref{lemma:main2}. 
   \begin{lemma}{\cite[Lemma 2]{kato_general})}\label{lem:kato2}
   With notations above,
   \begin{align*}
       U^{i,j}K_2(\mathcal{K}) \subseteq U^{i+j}K_{2}(\mathcal{K})
   \end{align*}
   More precisely, suppose $x \in \mathfrak{m}_{\mathcal{K}}^i$, $y \in \mathfrak{m}_{\mathcal{K}}^j$ for some $i,j \geq 1$ and $x \neq 0$. Then 
       \begin{align*}
       \{1+x, 1+y\} &\equiv \{1+xy, -x^{-1}\} {~~~mod~~} U^{i+j+1}K_2(\mathcal{K})
       \end{align*}
   \end{lemma}

   \begin{cor}\label{cor:blochkato}
       Let $i,j \in \mathbb{Z}$ be such that $i+j+1 >e_{\mathcal{K}}'$.  Suppose $x \in \mathfrak{m}_{\mathcal{K}}^i$, $y \in \mathfrak{m}_{\mathcal{K}}^j$ and $x \neq 0$. Then in $K_2(\mathcal{K})/p$,
       \begin{align*}
       \{1+x, 1+y\} = \{1+xy, -x^{-1}\}  
       \end{align*}
   \end{cor}
   \begin{proof}
       This follows from  Lemma \ref{lem:kato2} and the fact that $U^{i+j+1}K_2(\mathcal{K})/p =0$ (Lemma \ref{lem:unitpth}).
       \end{proof}

         For an $n$-tuple $\mathbf{u} = (u_1, u_2, \cdots, u_n) \in \mathcal{O}_\mathcal{K}^{\oplus n}$, let $\overline{\mathbf{u}} = (\overline{u}_1, \overline{u}_2, \cdots, \overline{u}_n) \in \kappa^{\oplus n}$. Given  $\mathbf{s} = (s_1, s_2,\dots s_n)\in \mathbb{F}_p^{\oplus n}$, let $\mathbf{u}^{\mathbf{s}}$ (resp.  $\overline{\mathbf{u}}^{\mathbf{s}}$) denote the product $u_1^{s_1}u_2^{s_2}\cdots u_n^{s_n} \in \mathcal{O}_\mathcal{K}$ (resp. $\overline{u}_1^{s_1}\overline{u}_2^{s_2}\cdots \overline{u}_n^{s_n} \in \kappa$). Let $\ceil{.}$ denote the ceil function.
\begin{lemma}\label{lem:unit}
    Let $\tilde{\mathcal{B}} = \{u_1, u_2, \cdots, u_n\}$ be a lift of some $p$-basis of $\kappa$ and let $\pi$ denote  a uniformizer of $\mathcal{K}$. Then any $\alpha \in \mathcal{O}_{\mathcal{K}}^{\times}$ is of the form
    \begin{align}\label{eqn:alpha}
        \alpha = \sum_{j=0}^{\ceil{e'_\mathcal{K}}}\sum_{\mathbf{s} \in \mathbb{F}_p^n } a_{\mathbf{s}}^p\mathbf{u}^{\mathbf{s}} \pi^j (mod~ U_{\mathcal{K}}^{\ceil*{e'_{\mathcal{K}}}+1})
    \end{align}
    for some $a_{\mathbf{s}} \in \mathcal{O}_{\mathcal{K}}$.
\end{lemma}
\begin{proof}
     Let 
 \begin{align*}
 \overline{\alpha} = \sum_{\mathbf{s} \in \mathbb{F}_p^n } \overline{a}_{1\mathbf{s}}^{p}\mathbf{\overline{u}}^{\mathbf{s}} ,\text{~~~~~~~~}\overline{a}_{1\mathbf{s}} \in \kappa
 \end{align*}
 For every $\mathbf{s}$, let $a_{1\mathbf{s}}$ denote a lift of $\overline{a}_{1\mathbf{s}}$ in $\mathcal{O}_{\mathcal{K}}$. Then we can write $\alpha$ as 
 \begin{align*}
     \alpha  =  \sum_{\mathbf{s} \in \mathbb{F}_p^n } a_{1\mathbf{s}}^{p}\mathbf{u}^{\mathbf{s}} + \alpha_1\pi^{n_1}
 \end{align*}
 for some $\alpha_1 \in \mathcal{O}_{\mathcal{K}}^{\times}$ and $n_1 >0$.  Repeating the above step, we can $\alpha_1$ as 
 \begin{align*}
          \alpha_1  =  \sum_{\mathbf{s} \in \mathbb{F}_p^n } a_{2\mathbf{s}}^{p}\mathbf{u}^{\mathbf{s}} + \alpha_2 \pi^{n_2}
 \end{align*}
 for some $\alpha_2 \in \mathcal{O}_{\mathcal{K}}^{\times}$, $a_{2\mathbf{s}} \in \mathcal{O}_{\mathcal{K}}$  and $n_2 >0$. Similarly express $\alpha_2$ as above and repeat the process for sufficiently many steps.
 Then we plug-in the expression for $\alpha_{i+1}$ in the previous expression for $\alpha_i$ recursively to get 
 \begin{align*}
        \alpha = \sum_{j=0}^{\ceil{e'_\mathcal{K}}}\sum_{\mathbf{s} \in \mathbb{F}_p^n } a_{\mathbf{s}}^{p}\mathbf{u}^{\mathbf{s}} \pi^j  + \gamma
    \end{align*}
for some $a_{\mathbf{s}} \in \mathcal{O}_{\mathcal{K}}$ and $\gamma \in \mathfrak{m}_{\mathcal{K}}^{\ceil*{e'_\mathcal{K}}+1}$. Since $\alpha \in \mathcal{O}_{\mathcal{K}}^{\times}$, 
\begin{align*}
    \sum_{j=0}^{\ceil{e'_\mathcal{K}}}\sum_{\mathbf{s} \in \mathbb{F}_p^n } a_{\mathbf{s}}^{p}\mathbf{u}^{\mathbf{s}} \pi^j \in \mathcal{O}_{\mathcal{K}}^{\times}
\end{align*}
and we get 
\begin{align*}
       \alpha &= (\sum_{j=0}^{\ceil{e'_\mathcal{K}}}\sum_{\mathbf{s} \in \mathbb{F}_p^n } a_{\mathbf{s}}^{p}\mathbf{u}^{\mathbf{s}} \pi^j)  (1+ \frac{\gamma}{\sum_{j=0}^{\ceil{e'_\mathcal{K}}}\sum_{\mathbf{s} \in \mathbb{F}_p^n } a_{\mathbf{s}}^{p}\mathbf{u}^{\mathbf{s}} \pi^j} ) \\
       &=\sum_{j=0}^{\ceil{e'_\mathcal{K}}}\sum_{\mathbf{s} \in \mathbb{F}_p^n } a_{\mathbf{s}}^p\mathbf{u}^{\mathbf{s}} \pi^j (mod~ U_{\mathcal{K}}^{\ceil*{e'_{\mathcal{K}}}+1})
\end{align*}
\end{proof}

\begin{lemma} \label{lem:pth}
     Let $\gamma\in \mathcal{O}_{\mathcal{K}}^{\times}$  be a finite sum of $p$-th powers  in $\mathcal{O}_{\mathcal{K}}$, i.e.,
    \begin{align*}
        \gamma = \sum_i b_i^p
    \end{align*}
    for some $b_i \in \mathcal{O}_{\mathcal{K}}$. Then 
    \begin{align*}
        \gamma \in (\mathcal{O}_{\mathcal{K}}^{\times})^p (mod ~ U_{\mathcal{K}}^{e_\mathcal{K}})
    \end{align*}
\end{lemma}
\begin{proof}
    This is because
    \begin{align*}
      \gamma=  \sum_i b_i^p &\in (\sum_i b_i)^p - p\mathcal{O}_{\mathcal{K}}\\
              &\in (\sum_i b_i)^p + \mathfrak{m}_{\mathcal{K}}^{e_\mathcal{K}}\\
        &\in (\sum_i b_i)^p (1 + \frac{1}{(\sum_i b_i)^p}\mathfrak{m}_{\mathcal{K}}^{e_\mathcal{K}})\\
        &\in (\mathcal{O}_{\mathcal{K}}^{\times})^p (mod ~ U_{\mathcal{K}}^{e_\mathcal{K}})
    \end{align*}
\end{proof}
Recall the definition $\iota_p^q$ from (\ref{eqn:iota1}) and (\ref{eqn:iota2}) in \S\ref{sec:kato}.
\begin{lemma}\label{lem:tame2}
   Suppose $p= char~\kappa =2$  and $2\mid e_{\mathcal{K}}$. Let $\alpha, \beta \in \mathcal{O}_{\mathcal{K}}^{\times}$  be a finite sum of squares  of elements  in $\mathcal{O}_{\mathcal{K}}$.  Then $\{\alpha, \beta\} \in K_2(\mathcal{K})/2$ is in the image of $\iota_2^2$ defined in (\ref{eqn:iota1}). 
\end{lemma}
\begin{proof}
    Let $\delta$ be a uniformizer of $\mathcal{K}$ and let
    \begin{align*}
        \alpha = \sum_i a_i^2\\
        \beta = \sum_i b_i^2
    \end{align*}
    be finite sums for some $a_i, b_i \in \mathcal{O}_{\mathcal{K}}$.  Following the arguments similar to the proof of Lemma \ref{lem:pth},  we get
    \begin{align*}
      \alpha = \sum_i a_i^2 &= (\sum_i a_i)^2 - 2a', {~~~~~~~~}a'\in \mathcal{O}_{\mathcal{K}}\\
      &=    (\sum_i a_i)^2 (1 + 2a''), {~~~~~~~}a''\in \mathcal{O}_{\mathcal{K}}
      \end{align*}
      Similarly we have
      \begin{align*}
          \beta = (\sum_i b_i)^2 (1 + 2b''),  {~~~~~~~~~}b''\in \mathcal{O}_{\mathcal{K}}
      \end{align*}
    Therefore
    \begin{align*}
        \{\alpha, \beta\}= \{1 + 2a'',1 + 2b''\}
    \end{align*}
    Suppose $a''b'' \in \mathfrak{m}_{\mathcal{K}}$, then by  Lemma \ref{lem:unitpth} and Lemma \ref{lem:kato2}, 
    \begin{align*}
    \{\alpha, \beta\} \in U_{\mathcal{K}}^{2e_{\mathcal{K}}+1}K_2(\mathcal{K}) =0 
    \end{align*}
   So the claim is  obviously true.\\
   \indent  Suppose $a''b'' \notin \mathfrak{m}_{\mathcal{K}}$, by Corollary \ref{cor:blochkato},
    \begin{align*}
        \{\alpha, \beta\} &= \{1 + 4u, -(2a'')^{-1}\}\\
        &= \{1 + (-2)^2u, u'\}, {~~~~~~~~~~} u,u' \in \mathcal{O}_{\mathcal{K}}^{\times}
    \end{align*}
    The last equality follows from the hypothesis that $2\mid e_{\mathcal{K}}$. Now it follows from Remark \ref{rmk:tame} that  $\{\alpha, \beta\} \in \iota_2^2(H^2_{\acute{e}t}(\kappa, \mathbb{Z}/2(1)))$. 
    \end{proof}
    For a complete discrete valued field $\mathcal{K}$, let $\mathcal{K}_{nr}$ denote the maximal unramified extension of $\mathcal{K}$. The following is the key lemma to prove Theorem \ref{thm:main}.
\begin{lemma}[Key Lemma]\label{lemma:main2}
    Let $i\geq 2$ and $\mathcal{K}$ be a complete discrete valued field with residue  field $\kappa$ of characteristic $p$. Suppose $\mathfrak{w} =\{w_1,w_2, \cdots, w_i\} \in K_i(\mathcal{K})/p$ be such that each $w_i \in \mathcal{O}_{\mathcal{K}}^{\times}$ and  is a finite sum of $p$-th powers in $\mathcal{O}_{\mathcal{K}}$. Then $\mathfrak{w} = 0$ if $p\neq 2$. If $p=2$, then  $\mathfrak{w} = 0$  for $i\geq 3$  and also for $i=2$ if any of the following conditions hold:
    \begin{enumerate}[(i)]
        \item $Br(\mathcal{K}_{nr}/\mathcal{K})[2] =0$
        \item $Br(\kappa)[2] =0$ and $2 \mid e_{\mathcal{K}}$.
    \end{enumerate}    
\end{lemma}
\begin{proof}
    By Lemma \ref{lem:pth}, 
    \begin{align*}
        w_i \in (\mathcal{O}_{\mathcal{K}}^{\times})^p (mod ~ U_{\mathcal{K}}^{e_\mathcal{K}})
    \end{align*}
    for each $i$. First, consider the case $i=2$. In this case,
\begin{align*}
    \mathfrak{w}  \in  U^{e_{\mathcal{K}}, e_{\mathcal{K}}} K_2(\mathcal{K})/p
    \end{align*}
Applying Lemma \ref{lem:kato2},  we see that 
    \begin{align*}
       \mathfrak{w}  \in U^{2e_{\mathcal{K}}}K_2(\mathcal{K})/p
    \end{align*}
   Now, when $i\geq 2$,  we can recursively use Lemma \ref{lem:kato2}  on the symbol $\mathfrak{w}$ as above  to get
    \begin{align*}
        \mathfrak{w} \in U^{ie_{\mathcal{K}}}K_i(\mathcal{K})/p
    \end{align*}
 If  $p\geq 3$, note that 
 \begin{align*}
     ie_{\mathcal{K}}>e'_{\mathcal{K}}, 
 \end{align*}
for all $i\geq 2$, and therefore by Lemma \ref{lem:unitpth}, we conclude that $ \mathfrak{w} =0$. \\
\indent Suppose  $p=2$ and $i\geq 3$. Then  again
    \begin{align*}
     ie_{\mathcal{K}}>e'_{\mathcal{K}}
 \end{align*}
and therefore by Lemma \ref{lem:unitpth}, we conclude that $\mathfrak{w} =0$.\\
\indent Let's now consider the case $p=2$ and $i=2$.  
Note that in this case, $\mathfrak{w}  \in U^{2e_{\mathcal{K}}}K_2(\mathcal{K})/p  = U^{e_{\mathcal{K}}'}K_2(\mathcal{K})/p\cong Br(\mathcal{K}_{nr}/\mathcal{K})[2]$ (Remark \ref{rmk:tame}). This proves (i). The  claim in (ii) follows from Lemma \ref{lem:tame2}.
\end{proof}

We are now ready to prove Theorem \ref{thm:main}:
\begin{proof}[Proof of Theorem \ref{thm:main}]
Suppose we are given $\mathfrak{w} =\{w_1,w_2, \cdots, w_i\} \in K_i(\mathcal{K})/p$. Let $\pi$ denote a uniformizer of $\mathcal{K}$ and let $\{u_1, u_2, \cdots, u_n\}$ denote a lift of  some $p$-basis of $\kappa$. Then $\Lambda =\{u_1, u_2, \cdots, u_n, \pi\}$ is a pseudo-basis of $\mathcal{K}$. Let
    \begin{align*}
        \mathcal{L}= \mathcal{K}(\Lambda^{1/p}) \in \mathfrak{PP}(\mathcal{K})
    \end{align*}
    Multiplying each $w_k$ by a suitable $r$-th power of uniformizer for some $r$ with $p|r$, we may assume without loss of generality that $w_k \in  \mathcal{O}_{\mathcal{K}}$ for all $k$.  Write  $w_k = v_k\pi^{m_k}$  for some $v_k \in \mathcal{O}_{\mathcal{K}}^{\times}$ and $m_k \geq 0$. Then  
    \begin{align*}
        \mathfrak{w}= \{v_1,v_2, \cdots, v_i \} + \Gamma
    \end{align*}
    where  $\Gamma$ is trivial over $\mathcal{K}(\sqrt[p]{\pi}) \subseteq \mathcal{L}$  and hence $Res^i_{\mathcal{L}/\mathcal{K}}(\Gamma) =0$.  So we may assume that $w_k \in \mathcal{O}_{\mathcal{K}}^{\times},  \forall k$. By Lemma \ref{lem:unit}, we have  for each index $k$,
    \begin{align*}
        w_k  &= w_k' \cdot (\sum_{j=0}^{\ceil{e'_\mathcal{K}}}\sum_{\mathbf{s} \in \mathbb{F}_p^n } w_{k\mathbf{s}}^{p}\mathbf{u}^{\mathbf{s}} \pi^j)
    \end{align*}
    for some  $w_{k\mathbf{s}} \in \mathcal{O}_\mathcal{K}$ and $w_k' \in U_{\mathcal{K}}^{\ceil*{e_{\mathcal{K}}'}+1}$.   Since $U_{\mathcal{K}}^{\ceil*{e'_{\mathcal{K}}}+1}\in (\mathcal{O}_{\mathcal{K}}^{\times})^p$ by Lemma \ref{lem:unitpth}, 
    \begin{align*}
        \mathfrak{w} = \{\sum_{j=0}^{\ceil{e'_\mathcal{K}}}\sum_{\mathbf{s} \in \mathbb{F}_p^n } w_{1\mathbf{s}}^{p}\mathbf{u}^{\mathbf{s}} \pi^j,  \sum_{j=0}^{\ceil{e'_\mathcal{K}}}\sum_{\mathbf{s} \in \mathbb{F}_p^n } w_{2\mathbf{s}}^{p}\mathbf{u}^{\mathbf{s}} \pi^j , \cdots, \sum_{j=0}^{\ceil{e'_\mathcal{K}}}\sum_{\mathbf{s} \in \mathbb{F}_p^n } w_{i\mathbf{s}}^{p}\mathbf{u}^{\mathbf{s}} \pi^j  \}
    \end{align*}
    We now note that when we base change to $\mathcal{L}$, each of the summands in the above expression becomes sum of $p^{th}$ powers  of elements in $\mathcal{O}_\mathcal{L}$.  Moreover, note that by the construction of $\mathcal{L}$,   $p \mid e_{\mathcal{L}}$.  Let $\mathfrak{l}$ denote the residue field of $\mathcal{L}$. The claim now follows from Lemma \ref{lemma:main2} together with the observation that $Br(\mathfrak{l})[2] \cong H^2_{\acute{e}t}(\mathfrak{l}, \mathbb{Z}/2(1)) =0$ since   $\mathfrak{l}/\kappa$ is purely inseparable (\cite[Chapter IV, Theorem 4.1.5]{jacobson_fin_dim}). 
\end{proof}

\section{When the residual characteristic is two} \label{sec:char2}
With notations as before, let $\mathcal{L} \in \mathfrak{PP}(\mathcal{K})$. We provide an example to illustrate  that the natural restriction map
  \begin{align*}
 Res_{\mathcal{L}/\mathcal{K}}^2: K_2(\mathcal{K})/2 \rightarrow K_2(\mathcal{L})/2
\end{align*}
need not be zero  if $p=char~\kappa = 2$  and $H^2_{\acute{e}t}(\kappa,\mathbb{Z}/2(1)) \neq 0$.\\
\indent Let $\kappa = \mathbb{F}_2(\mathfrak{a},\mathfrak{b})$ where $\mathbb{F}_2$  denotes the finite field with two elements and $\mathfrak{a},\mathfrak{b}$ are purely transcendental over $\mathbb{F}_2$. Let $\mathcal{K}$ be a complete discrete valued field of characteristic zero with residue $\kappa$ and uniformizer $p=2$.  Such a field exists by \cite[\href{https://stacks.math.columbia.edu/tag/0328}{Tag 0328}]{stacks-project}.  Let $a,b$ be lifts of $\mathfrak{a},\mathfrak{b}$ in $\mathcal{O}_{\mathcal{K}}$. Consider the symbol 
    \begin{align*}
        \{1+a, 1+b\} \in K_2(\mathcal{K})/2
    \end{align*}
    Let $\mathcal{L} = \mathcal{K}(\sqrt{a}, \sqrt{b}, \sqrt{2}) \in \mathfrak{PP}(\mathcal{K})$. Then $\sqrt{2}$ is a uniformizer of $\mathcal{L}$ with residue field $\mathfrak{l} = \mathbb{F}_2(\sqrt{\mathfrak{a}}, \sqrt{\mathfrak{b}})$. Now, 
     \begin{align*}
        \{1+a, 1+b\} \otimes_{\mathcal{K}} \mathcal{L}  &= \{1+\sqrt{a}^2, 1+\sqrt{b}^2\}\\
        &=\{(1+\sqrt{a})^2 - 2 \sqrt{a}, (1+\sqrt{b})^2 - 2 \sqrt{b}\}\\
        &= \{1- 2\sqrt{a}(1+\sqrt{a})^{-2}, 1- 2\sqrt{b}(1+\sqrt{b})^{-2})\}\\
        &=\{1+ 4\sqrt{a}\sqrt{b}(1+\sqrt{a})^{-2}(1+\sqrt{b})^{-2}, (\sqrt{2})^{-2}(\sqrt{a})^{-1}(1+\sqrt{a})^{2}\} \text{~~~~~~~(Corollary \ref{cor:blochkato})}\\
        &=\{1+ 4\sqrt{a}\sqrt{b}(1+\sqrt{a})^{-2}(1+\sqrt{b})^{-2}, \sqrt{a}\} 
 \end{align*}
  By Remark \ref{rmk:tame} and Remark \ref{rmk:brauerp}, the map $\gamma_{\sqrt{2}}^{-1}$  takes the above element to the symbol 
 \begin{align*}
 \alpha = [\sqrt{\mathfrak{a}}\sqrt{\mathfrak{b}}(1+\sqrt{\mathfrak{a}})^{-2}(1+\sqrt{\mathfrak{b}})^{-2}, \sqrt{\mathfrak{a}})  \in  Br(\mathfrak{l})[2] \cong H^2_{\acute{e}t}(\mathfrak{l},\mathbb{Z}/2(1))  
 \end{align*}

 We now show that $ \alpha \neq 0$ which implies $\{1+a, 1+b\}\otimes_{\mathcal{K}} \mathcal{L} \neq 0$.\\
\indent To simplify notation, let $\sqrt{\mathfrak{a}} = v$ and $\sqrt{\mathfrak{b}} =u$  so that $\mathfrak{l} \cong \mathbb{F}_2(u,v)$ as abstract fields where $u,v$ are purely transcendental over $\mathbb{F}_2$. Then
\begin{align*}
\alpha = [uv (1+u)^{-2}(1+v)^{-2}, v)  \in H^2_{\acute{e}t}(\mathfrak{l},\mathbb{Z}/2(1)) \simeq Br(\mathfrak{l})[2]
\end{align*}
Let $\kappa'/\mathfrak{l}$ be the Artin-Schreier extension  given by 
\begin{align*}
    \kappa' \simeq \mathfrak{l}[x]/(x^2+ x - uv (1+u)^{-2}(1+v)^{-2})
\end{align*}
 Then $\alpha$ is trivial if and only if $v \in Norm_{\kappa'/\mathfrak{l}}(\gamma)$ for some  $\gamma \in \kappa'$ (\cite[Chapter 4, Corollary 4.7.5]{GilleSzamuely:2006}). This happens if and only if the polynomial 
\begin{align*}
    z^2 + zw+ uv (1+u)^{-2}(1+v)^{-2}w^2 -v \in \mathfrak{l}[z,w]
\end{align*}
  has a root. Using the change of variables via 
\begin{align*}
    \tilde{z} = z(1+u)(1+v)
\end{align*}
the above polynomial has a root in $\mathfrak{l}$ if and only if 
\begin{align*}
    \tilde{z}^2 +(1+u)(1+v)\tilde{z}w+ uv w^2 = v (1+u)^2(1+v)^2
\end{align*}
has a solution in $\mathfrak{l}$.  Write $\tilde{z} = f/g'$ and $w =h/g'$  for some $f,g',h \in \mathbb{F}_2[u,v], g'\neq 0$. Let $g= g'(1+u)(1+v)$. Then it suffices to show that 
\begin{align*}
    f^2 + (1+u)(1+v)fh + uvh^2 = vg^2
\end{align*}
has no solution in $\mathbb{F}_2[u,v]$.  This follows from the  following proposition.
\begin{prop}
    Let $k$ be a field of characteristic $2$ and let  $\Omega = k(v)$ be a purely transcendental extension of $k$. There exists no $f,g,h \in \Omega[u] $ such that 
    \begin{align}\label{eqn:norm}
    f^2 + (1+u)(1+v)fh + uvh^2 =vg^2
   \end{align}
\end{prop}
\begin{proof}
    Suppose there exists such  $f,g,h \in \Omega[u] $.  Let $q=gcd(f,h)$. Then it is easy to see that $q^2 \mid g^2$ and hence $q\mid g$. Dividing the above equation by $q^2$, we may assume that $gcd(f,h) =1$. Write
    \begin{align*}
        f&= (1+u)\tilde{f} + r_1\\
        g&= (1+u) \tilde{g} + r_2\\
        h &= (1+u) \tilde{h} + r_3
    \end{align*}
   where $r_1, r_2, r_3 \in \Omega$.  Since $(1+u) \mid f^2 +vg^2 + uvh^2$, we see that $(1+u) \mid (r_1^2 + vr_2^2) + u (vr_3^2)$. Therefore, 
   \begin{align*}
   r_1^2 + vr_2^2 = vr_3^2
   \end{align*}
Since $v$ is not a square in $\Omega$, $r_1=0$ and $r_2=r_3$. This implies that $(1+u) \mid f$. Now by (\ref{eqn:norm}), $(1+u^2) \mid (g^2 + uh^2)$ so that $(1+u^2) \mid (r_2^2 + ur_3^2)$. Hence $r_2 =r_3 =0$. 
In particular, $(1+u) \mid h$. Thus $(1+u) \mid gcd(f,h)$ contradicting the assumption that $gcd(f,h) =1$. 
\end{proof}

The above arguments imply the following theorem:
\begin{thm}\label{thm:char2}
Let $\mathcal{K}$ be a complete discrete valued field of characteristic zero  with uniformizer $2$ and residue $\kappa \simeq \mathbb{F}_2(\mathfrak{a}, \mathfrak{b})$ where $\mathfrak{a}, \mathfrak{b}$ are purely transcendental over $\kappa$. Let $\mathcal{L} \simeq \mathcal{K}(\sqrt{a}, \sqrt{b}, \sqrt{2}) \in \mathfrak{PP}(\mathcal{K})$. Then  the natural restriction map
  \begin{align*}
 Res_{\mathcal{L}/\mathcal{K}}^2: K_2(\mathcal{K})/2 \rightarrow K_2(\mathcal{L})/2
\end{align*}
is not zero.  
\end{thm}

\begin{remark}\normalfont
The above example illustrates that  unlike the $p\neq 2$ case,  when $p=2$  the collection  $\mathfrak{PP}(\mathcal{K})$  can possibly contain fields that does not split every class in $Br(\mathcal{K})[2]$. However, this does not imply that Conjecture \ref{conj:main} is false for $p=2$. We also note that  by \cite[Theorem 4.16]{nivedita}, the conjecture is true whenever $\mathscr{R}_p(\kappa) \leq 3$. So to check the validity of the conjecture, one can assume $\mathscr{R}_p(\kappa)\geq 4$. 
\end{remark}

\begin{question}\normalfont
  Let $\mathcal{K}$ be a complete discrete valued field of characteristic zero and residue $\kappa= \mathbb{F}_2(\tilde{u_1}, \tilde{u_2}, \tilde{u_3}, \tilde{u_4})$ where $\tilde{u_i}$'s are transcendental. Such a field exists by \cite[\href{https://stacks.math.columbia.edu/tag/0328}{Tag 0328}]{stacks-project}.   Let $u_i$ be lifts of $\tilde{u_i}$ in $\mathcal{O}_{\mathcal{K}}$  and let  $\pi$ be a uniformizer of $\mathcal{K}$.  Consider the following central simple algebra over $\mathcal{K}$:
  \begin{align*}
      A = (u_1,  u_2+ u_3) \otimes_{\mathcal{K}} (u_2, u_3 + u_4) \otimes_{\mathcal{K}} (u_3, u_4 + u_1) \otimes_{\mathcal{K}} (u_4 , u_1 + u_2) \otimes_{\mathcal{K}} (u_1 + u_3, u_2 + u_4) \otimes_{\mathcal{K}} (1-u_1, \pi)
 \end{align*}
Is $A$ a division algebra? An affirmative answer to the question would mean $ind(A) = 2^6$, violating Conjecture \ref{conj:main}. 
\end{question}

\section{Applications to period-index problems} \label{sec:applications}
\subsection{Brauer $p$-dimension of complete discrete valued fields}\label{sec:mainh}
As before, let $\mathcal{K}$ denote a complete discrete valued field with residue $\kappa$ of characteristic $p$.  In this section, we prove Corollary \ref{cor:mainh} which says that the assertions of  Theorem \ref{thm:main} hold if we replace the  mod-$p$ Milnor $K$-theory with  $H^i(\mathcal{K}, \mu_p^{\otimes i-1})$ without requiring that the primitive $p$-th  roots of unity are in $\mathcal{K}$ i.e., we show that  for any $\mathcal{L} \in \mathfrak{PP}(\mathcal{K})$
 \begin{align*}
   Res_{\mathcal{L}/\mathcal{K}}^i:     H^i(\mathcal{K}, \mu_p^{\otimes i-1}) \rightarrow H^i(\mathcal{L}, \mu_p^{\otimes i-1})
        \end{align*}
  is zero if
  \begin{align*}
Res_{\mathcal{L}/\mathcal{K}}^i: K_i(\mathcal{K})/p \rightarrow K_i(\mathcal{L})/p
\end{align*}
is zero.
  
The case $p=2$ directly follows from the norm residue isomorphism (Theorem \ref{thm:voe}) and from the naturality of the Galois symbol.  It is also true for $p\neq 2$ as we show below. The idea is to first  assume the presence of primitive $p$-th roots of unity in $\mathcal{K}$ so that we can  apply the norm residue isomorphism theorem and  then use a standard restriction corestriction argument to show that the assertions hold even without the assumption on the presence of roots of unity. We begin with the following lemma:
\begin{lemma}\label{lem:unity}
Let $\mathcal{E}/\mathcal{K}$ be a finite extension of degree coprime to $p$. Suppose $\mathcal{L} \in \mathfrak{PP}(\mathcal{K})$. Then $\mathcal{L} \mathcal{E} \in \mathfrak{PP}(\mathcal{E})$.
\end{lemma}

\begin{proof}
Let 
\begin{align*}
\mathcal{L} = \mathcal{K}(\sqrt[p]{\pi}, \sqrt[p]{u_1}, \sqrt[p]{u_2},\cdots, \sqrt[p]{u_n})
\end{align*}
where  $\pi$ is a uniformizer of $\mathcal{K}$ and $\{u_1, u_2, \dots, u_n\}$ is a lift of some $p$-basis of $\kappa$. Then 
\begin{align*}
\mathcal{L} \mathcal{E} = \mathcal{E}(\sqrt[p]{\pi}, \sqrt[p]{u_1}, \sqrt[p]{u_2},\cdots, \sqrt[p]{u_n})
\end{align*}
Let $\mathfrak{e}$, $\delta$  denote  respectively the residue  and a uniformizer of $\mathcal{E}$.  Note that since $[\mathcal{E}:\mathcal{K}]$ is coprime to $p$,  $[\mathfrak{e}:\kappa]$  as well as the ramification index $e_{\mathcal{E}/\mathcal{K}}$ are coprime to $p$. So any $p$-basis of  $\kappa$ is still a $p$-basis of $\mathfrak{e}$  and $\mathcal{E}(\sqrt[p]{\pi}) = \mathcal{E}(\sqrt[p]{u\delta^{ e_{\mathcal{E}/\mathcal{K}}}}) = \mathcal{E}(\sqrt[p]{v\delta})$ for some $u, v \in \mathcal{O}_{\mathcal{E}}^{\times}$.  Since  $v\delta$ is also a uniformizer of $\mathcal{E}$, we conclude  from the above discussion that $\mathcal{L} \mathcal{E} \in \mathfrak{PP}(\mathcal{E})$. 
\end{proof}

\begin{proof}[Proof of Corollary \ref{cor:mainh}]
    Let $\zeta$ denote a primitive $p$-th root of unity. First assume that $\mathcal{K}$ contains $\zeta$.  By the norm residue isomorphism theorem (Theorem \ref{thm:voe}), 
    \begin{align*}
        K_i(\mathcal{K})/p \simeq H^i(\mathcal{K}, \mu_p^{\otimes i}) \simeq H^2(\mathcal{K}, \mu_p)\simeq H^i(\mathcal{K}, \mu_p^{\otimes i-1}) 
    \end{align*}
    By the functoriality of the Galois symbol and by Theorem \ref{thm:main}, it is easy to see that the natural restriction map
\begin{align}\label{eqn:res}
Res^i_{\mathcal{L}/\mathcal{K}}: H^i(\mathcal{K}, \mu_p^{\otimes i-1}) \rightarrow H^i(\mathcal{L}, \mu_p^{\otimes i-1})
\end{align}
is zero.  Now suppose that  $\mathcal{K}$ does not contain $\zeta$. The composite 
\begin{align*}
    H^i(\mathcal{K}, \mu_p^{\otimes i-1}) \rightarrow  H^i(\mathcal{K}(\zeta) ,\mu_p^{\otimes i-1})  \rightarrow H^i(\mathcal{L}(\zeta), \mu_p^{\otimes i-1})
\end{align*}
    is zero by (\ref{eqn:res}) and Lemma \ref{lem:unity}.  Since $[\mathcal{K}(\zeta):\mathcal{K}]$ is coprime to $p$ and $[\mathcal{L}:\mathcal{K}] = p^{n+1}$,  we have $[\mathcal{L}\mathcal{K}(\zeta):K] = [\mathcal{L}:K][\mathcal{K}(\zeta):K]$. Therefore, $\mathcal{L}$ and $\mathcal{K}(\zeta)$ are linearly disjoint over $\mathcal{K}$ (\cite[Corollary 2.5.2]{fried}). So
    \begin{align*}
        \mathcal{L}(\zeta) = \mathcal{L} \otimes_{\mathcal{K}} \mathcal{K}(\zeta)
    \end{align*}
    and the composition
    \begin{align*}
        H^i(\mathcal{K}, \mu_p^{\otimes i-1}) \xrightarrow{Res_{\mathcal{L}/\mathcal{K}}} H^i(\mathcal{L}, \mu_p^{\otimes i-1}) \rightarrow H^i(\mathcal{L}(\zeta), \mu_p^{\otimes i-1})
    \end{align*}
    is zero.  But $[\mathcal{L}(\zeta): \mathcal{L}]$ is coprime to $p$, so  the second arrow is injective by a standard restriction corestriction argument.  Therefore $Res_{\mathcal{L}/\mathcal{K}} =0$.
\end{proof}

The bound for Brauer $p$-dimension of $\mathcal{K}$  (Conjecture \ref{conj:main}) easily follows from the above discussions:
\begin{align*}
   Br_pdim(\mathcal{K}) \leq \mathscr{R}_{ps}(\mathcal{K})
    \end{align*}

\subsection{Period-index bounds for complete regular local rings}\label{sec:higher}
Throughout this section, $\mathcal{K}$ denotes the fraction field of a complete regular local ring $\mathcal{R}$ with residue field $\kappa$ of characteristic $p$. Let $\{u_1, u_2, \cdots, u_n\}$ denote a $p$-basis of $\kappa$.  In this section, we prove Theorem \ref{thm:ndim}.\\
\indent The proof  is a generalization of  the proof of \cite[Proposition 3.5, Theorem 3.6]{ps14} where the corresponding result is shown for $2$-dimensional complete regular local rings. However there is a  gap in the proof of \cite[Proposition 3.5]{ps14} which was pointed out by the author to V. Suresh who later sketched out a way to fix the gap.  We provide a proof of Theorem \ref{thm:ndim} using these ideas.  To prove the theorem, we need the following lemma.
\begin{lemma} \label{lem:trivial}
    Let $\mathcal{L}$ be the fraction field of a complete regular  local ring $\mathcal{S}$ with maximal ideal $\mathfrak{m}$. Let $x\in Spec(\mathcal{S})$ be a regular prime and $\mathcal{L}_x$ be  the fraction field of the completion $\hat{\mathcal{S}}_x$ of $\mathcal{S}_x$. Suppose $\alpha \in Br(\mathcal{L})$ is unramified  on $\mathcal{S}$ and  is trivial over $\mathcal{L}_x$. Then $\alpha$ is trivial.
\end{lemma}
\begin{proof}
    Since $\alpha \in Br(\mathcal{L})$ is unramified on $\mathcal{S}$, there exists  $\mathfrak{A} \in Br(\mathcal{S})$ corresponding to the class of  an Azumaya algebra over $\mathcal{S}$ such that $\mathfrak{A} \otimes_\mathcal{S} \mathcal{L} = \alpha$. Since $\mathfrak{A} \otimes_\mathcal{S} \mathcal{L}_x$ is trivial and $\hat{\mathcal{S}}_x$ is a complete discrete valued ring, by \cite[Theorem 7.2]{aus_goldman},  $\mathfrak{A} \otimes_\mathcal{S} \hat{\mathcal{S}}_x$  is trivial. Therefore $\mathfrak{A} \otimes_\mathcal{S} \mathcal{S}_x/(x)$  is trivial. Since $\mathcal{S}_x /(x)$ is isomorphic to the fraction field of $\mathcal{S}/(x)$, and  $\mathcal{S}/(x)$  is a   regular local ring (\cite[\href{https://stacks.math.columbia.edu/tag/00NQ}{Tag 00NQ}]{stacks-project}), $\mathfrak{A}\otimes_\mathcal{S} \mathcal{S}/(x)$ is trivial (\cite[Theorem 7.2]{aus_goldman}). This implies that, $\mathfrak{A} \otimes_\mathcal{S} \mathcal{S}/\mathfrak{m}$ is trivial. But $\mathcal{S}$ is a complete regular local ring and therefore by \cite[Theorem 2.2]{saltman_br_p}, $\mathfrak{A}$ is trivial. Hence  $\alpha$ is trivial as well.
\end{proof}

\begin{proof} [Proof of Theorem \ref{thm:ndim}]
    Let $\alpha \in Br(\mathcal{K})[p]$ be such that it is ramified at most at some regular system of parameters $\Pi =\{\pi_1, \pi_2, \cdots, \pi_m\}$ of $\mathcal{K}$.  Let $\Lambda = \tilde{\mathcal{B}} \cup \Pi$ be a pseudo-basis of $\mathcal{K}$ where $\tilde{\mathcal{B}}$ is a lift of some $p$-basis $\mathcal{B}$ of $\kappa$. Let
    \begin{align*}
        \mathcal{L} = \mathcal{K}(\Lambda^{1/p}) \in \mathfrak{PP}(\mathcal{K})
    \end{align*}
   
    We need to show that $\alpha \otimes_{\mathcal{K}} \mathcal{L}$ is trivial. \\
    \indent Let $\mathcal{S} = \mathcal{R}[\Lambda^{1/p}]$. Then $\mathcal{S}$ is a complete regular local ring (\cite[Corollary 3.3]{ps14}, \cite[Theorem 7]{cohen_structure}),    with residue field $\kappa(\mathcal{B}^{1/p})$ and fraction field $\mathcal{L}$.   We now show that $\alpha$ is unramified over $\mathcal{S}$. Note that if $\mathcal{P}$ is a height one prime ideal in $\mathcal{R}$ that is not $(\pi_i)$ for some $i$,  then  by hypothesis $\alpha$ is unramified  at $\mathcal{P}$ and hence unramified over any height one prime ideal in $\mathcal{S}$ lying above $\mathcal{P}$.  We now divide the  proof into two cases: 
    \begin{enumerate}[(i)]
        \item Suppose for every $1\leq i \leq m$, $\mathcal{R}/(\pi_i)$ is of characteristic $\neq p$. In this case, $\mathcal{R}/(\pi_i)$  is necessarily of characteristic zero.  Let $\mathcal{L}' = \mathcal{K}(\Pi^{1/p})$. Note that  $\alpha \otimes_{\mathcal{K}} \mathcal{L}'$ is unramified at each $(\pi_i^{1/p})$  (\cite[Lemma 3.4]{ps14}). Let
        \begin{align*}
            \mathcal{S}' \simeq \mathcal{R}[\Pi^{1/p}]
        \end{align*}
         By purity for Brauer groups (Theorem \ref{thm:purity}), $\alpha \otimes_{\mathcal{K}} \mathcal{L}'$ is unramified on $S'$. Since $S'$ is a complete regular local ring with residue $\kappa$ (\cite[Corollary 3.3]{ps14}), it follows from  \cite[Theorem 7.2]{aus_goldman} that,
        \begin{align} \label{eqn:residue}
            Br(\mathcal{S}') \simeq Br(\kappa)
        \end{align}
         Since $Br(\kappa)[p]$ is generated by cyclic $p$-algebras (\cite[Chapter IV, Theorem 4.2.17]{jacobson_fin_dim}), the natural restriction
        \begin{align*}
            Br(\kappa)[p] \rightarrow Br(\kappa(\mathcal{B}^{1/p}))[p]
        \end{align*}
        is zero. But $\mathcal{S}$ is a complete regular local ring with residue field $\kappa(\mathcal{B}^{1/p})$, so by the naturality of the isomorphism in (\ref{eqn:residue}) with respect to restriction maps, we see that the restriction
        \begin{align*}
            Br(\mathcal{S}') \rightarrow Br(\mathcal{S})
        \end{align*}
        is zero. Hence $\alpha \otimes_{\mathcal{K}} \mathcal{L}$ is trivial.
        \item Suppose $\mathcal{R}/(\pi_i)$  is of characteristic $p$ for some $i$.   Without loss of generality, assume that $i=1$. Then $\mathcal{R}/(\pi_j)$ is necessarily of characteristic zero for $j \neq  1$. Let $\mathcal{Q}'$ be a height one prime ideal in $\mathcal{S}$ above $(\pi_j)$. Since $\mathcal{L}/\mathcal{K}$ is ramified over $(\pi_j)$ and $\mathcal{R}/(\pi_j)$ is of characteristic zero, $\alpha \otimes_{\mathcal{K}} \mathcal{L}$ is unramified at $\mathcal{Q}'$ (\cite[Lemma 3.4]{ps14}).  By \cite[\href{https://stacks.math.columbia.edu/tag/00NQ}{Tag 00NQ}]{stacks-project}, $\mathcal{R}/(\pi_1)$ is a complete regular local ring with  $\{\pi_2, \cdots,  \pi_m\}$ as a regular system of parameters and residue $\kappa$. Since  $\mathcal{R}/(\pi_1)$ is a $(p,p)$-equicharacteristic  complete  Noetherian local ring, by Cohen's structure theorem $\mathcal{R}/(\pi_1) \simeq \kappa[[\overline{\pi}_2, \cdots,  \overline{\pi}_m]]$ where $\overline{\pi}_k$ is the image of $\pi_k$ in $\mathcal{R}/(\pi_1)$.  Let $\kappa'$ be  the fraction field of $\mathcal{R}/(\pi_1)$. Then $\kappa'\simeq \kappa(( \overline{\pi}_2, \cdots,  \overline{\pi}_m))$ and $\{\overline{u}_1, \overline{u}_2,\cdots, \overline{u}_n, \overline{\pi}_2, \cdots,  \overline{\pi}_m\}$  form a $p$-basis of $\kappa'$.   Now  let $\mathcal{Q}$ be the height one prime ideal in $\mathcal{S}$ above $\pi_1$. Let $\mathcal{K}_{(\pi_1)}$  be the fraction field of  the completion $\hat{\mathcal{R}}_{(\pi_1)}$ and  let $\mathcal{L}_{\mathcal{Q}}$  be the fraction field of the completion $\hat{\mathcal{S}}_{(\mathcal{Q})}$.  We have 
        \begin{align*}
        \mathcal{L}_{\mathcal{Q}} \simeq \mathcal{K}_{(\pi_1)} (\Lambda^{1/p}) 
        \end{align*}
        Since $\mathcal{L}_{\mathcal{Q}}$ is a complete discrete   valued field  and the  residue field of  $\mathcal{K}_{(\pi_1)}$ is  $\kappa'$, $\mathcal{L}_{\mathcal{Q}} \in \mathfrak{PP}(\mathcal{K}_{(\pi_1)})$. Therefore, by Corollary \ref{cor:mainh}, $\alpha \otimes_{\mathcal{K}} \mathcal{L}_{\mathcal{Q}}$ is trivial. By \cite[Lemma 3.4]{ps14}, $\alpha \otimes_{\mathcal{K}} \mathcal{L}$ is unramified at $\mathcal{Q}$. Now  the above arguments together with the purity result for Brauer groups  (Theorem \ref{thm:purity}),  implies that $\alpha \otimes_{\mathcal{K}} \mathcal{L}$ is unramified on $\mathcal{S}$. By  taking $x=\mathcal{Q}$ in Lemma \ref{lem:trivial}, we deduce that $\alpha \otimes_{\mathcal{K}} \mathcal{L}$ is trivial.
        \end{enumerate}
\end{proof}

\begin{question}
    Let $\mathcal{K}$ be  the fraction field of a complete regular local ring  of characteristic zero  and dimension at least two with residue field  of characteristic $p \neq 2$. Suppose $\mathcal{L} \in \mathfrak{PP}(\mathcal{K})$. Is the natural restriction map
  \begin{align*}
Res_{\mathcal{L}/\mathcal{K}}^2: K_2(\mathcal{K})/p \rightarrow K_2(\mathcal{L})/p
\end{align*}
zero?
\end{question}

\subsection{Brauer $p$-dimension of semi-global fields}\label{sec:semiglobal}
In this section, $\mathcal{K}$ is a complete discrete valued field of characteristic zero with residue $\kappa$ of characteristic $p$. Assume that  $p=char~\kappa \neq2$ or that $H^2_{\acute{e}t}(\kappa,\mathbb{Z}/2(1))=0$ holds. We  prove Theorem \ref{thm:semiglobal} that if  $F$ is a semi-global field over $\mathcal{K}$ i.e., the function field of a curve over $\mathcal{K}$, then   
\begin{align*}
    Br_pdim(F) \leq \mathscr{R}_{ps}(\mathcal{K})+1
\end{align*}
The proof involves applying Theorem \ref{thm:main}, Theorem \ref{thm:ndim} and  patching techniques of  Harbater, Hartmann and Krashen (\cite{hhk09}).
\begin{proof}[Proof of Theorem \ref{thm:semiglobal}]
    The proof goes similar to \cite[Theorem 3.6]{ps14} with appropriate changes. Let $n=\mathscr{R}_p(\kappa)$. First, note that  since the transcendence degree of any finite extension of $F$ is one, it suffices to show that for any division algebra $D$ over $F$, $ind(D)\mid per(D)^{n+2}$. Moreover, by \cite[Corollary 1.1]{ps14}, we may assume that the period of $D$ is $p$. Since the $p$-rank of a field does not change under finite extensions, by replacing $\mathcal{K}$ by a finite extension, we may assume that $F$ is the function field of geometrically integral smooth projective curve over $\mathcal{K}$. \\
     \indent We now recall the standard patching setup  of \cite{hhk09}. Let $\pi$ denote a uniformizer of $\mathcal{O}_{\mathcal{K}}$. After sufficiently many blow-ups one can obtain a   regular proper model $\mathcal{X}$ of $F$ over $\mathcal{O}_{\mathcal{K}}$ such that  the support of the ramification divisor of $[D]$ and the components of the reduced special fiber are a union of regular curves with normal crossings on $\mathcal{X}$. For a closed point $P \in \mathcal{X}$, let $F_P$ denote the fraction field of the completion $\hat{\mathcal{O}}_{\mathcal{X},P}$. Let $\mathcal{X}_0$ be the special fiber of $\mathcal{X}$.  For an open subset $U$ of an irreducible component of $\mathcal{X}_0$, let $R_U$ denote the ring of regular functions on $U$. Since $\mathcal{O}_{\mathcal{K}}\subset R_U$, we can take the $\pi$-adic completion $\hat{R}_U$ of $R_U$. Denote the fraction field of $\hat{R}_U$ by $F_U$. For  an irreducible component $\eta$  of $\mathcal{X}_0$ let $F_{\eta}$ denote the completion of $F$ at the discrete valuation given by $\eta$. \\
     \indent Since the residue field of $F_{\eta}$ is of  transcendence degree one over $\kappa$, its $p$-rank is $n+1$. Therefore, by Corollary \ref{cor:mainh}, $ind(D \otimes_F F_{\eta})  \mid p^{n+2}$.  By \cite[Proposition 5.8]{hhk15} and \cite[Proposition 1.17]{boi}, there exists an irreducible open set $U_{\eta}$ of $\mathcal{X}_0$ containing $\eta$ such that $ind(D \otimes_F F_{U_{\eta}}) = ind(D \otimes_F F_{\eta})$ and hence $ind(D \otimes_{F} \mathcal{F}_{U_{\eta}}) \mid p^{n+2}$. Let $S_0$ be a finite set of closed points of $\mathcal{X}$ containing all points of intersection of components of $\mathcal{X}_0$ and the support of the ramification divisor of $D$. Let $S$ be a finite set of closed points of $\mathcal{X}$ containing $S_0$ and $\mathcal{X}_0\setminus (\bigcup U_{\eta})$ where $\eta$ varies  over generic points of $\mathcal{X}_0$. Then by \cite[Theorem 5.1]{hhk09}, 
     \begin{align}\label{eqn:ind}
         ind(D) = l.c.m\{ind(D\otimes_F F_{\zeta})\}
     \end{align}
     where $\zeta$ runs over $S$ and irreducible components of $\mathcal{X}_0 \setminus S$.\\
     \indent If $\zeta = U$ for some irreducible component $U$ of $\mathcal{X}_0\setminus S$, then $U \subset U_{\eta}$ and $F_{U_{\eta}} \subset F_U$, where  $\eta$ is the generic point of $U$.  Hence $ind(D\otimes _{\mathcal{F}} F_U) \mid p^{n+2}$.\\
     \indent On the other hand, if $\zeta = P \in S$,  the local ring $\mathcal{O}_P$  at $P$ is regular of dimension two and  by the construction of $\mathcal{X}$,  the ramification locus of $[D]$ is a subset of regular system of parameters of $\mathcal{O}_P$. Since the residue field of $\mathcal{O}_P$ is a finite extension of $\kappa$, its $p$-rank is also equal to $n$. Hence, Theorem \ref{thm:ndim} implies $ind(D\otimes_{F} F_P) \mid p^{n+2}$.  Finally, from   equation (\ref{eqn:ind}),  we get $ind(D) \mid p^{n+2}$ as required.
     \end{proof}

\subsection{Uniform bounds for semi-global fields}\label{sec:uniform}
\indent  We keep the notations in \S\ref{sec:semiglobal}. In this section, we prove Theorem \ref{thm:uniform} that 
\begin{align*}
    gssd_p^2(F)\leq 2(\mathscr{R}_{ps}(\mathcal{K})+1)
\end{align*}
 and that $F$ is uniformly $(2,p)$-bounded with  $[F(\zeta):F]p^{2(\mathscr{R}_{ps}(\mathcal{K})+1)}$  a $(2,p)$-uniform bound of $F$ (where $\zeta$ is a primitive $p$-th root of unity). The proof is similar to that in \cite[Theorem 2.4]{ps_uniform} with appropriate changes.

\begin{proof}[Proof of Theorem \ref{thm:uniform}]
    Since any finite extension of a semi-global field is semi-global, it suffices to show that for any finite set $S=\{\alpha_1, \alpha_2,\cdots, \alpha_m\} \in  Br(F)[p]$, $ind(S) | p^{2(\mathscr{R}_{ps}(\mathcal{K})+1)}$. As before, choose a regular proper model $\mathcal{X}$ of $F$ over $\mathcal{O}_{\mathcal{K}}$ such that the support of the ramification divisors of $\alpha_i$ for all $i$ and the special fiber is contained in $C\cup E$ where $C$ and $E$ are regular (not necessarily irreducible) curves on $\mathcal{X}$ with only normal crossings.\\
\indent Let $n=\mathscr{R}_p(\kappa)$. For any generic point $\eta$ of an irreducible component of the special fiber $\mathcal{X}_0$ of $\mathcal{X}$,   the residue  field $\kappa(\eta)$ of $F_{\eta}$ is of transcendence degree one  over $\kappa$ and therefore its $p$-rank is $n+1$. Let $\pi_{\eta}$ denote a uniformizer of $F_{\eta}$ and let $u_{\eta,1}, \dots, u_{\eta,{n+1}}$  denote a lift of some $p$-basis of $\kappa(\eta)$.  By the Chinese remainder theorem, there exists $f, u_1, \cdots, u_{n+1} \in F^{\times}$ such that $f$ is a uniformizer at each $\eta$ and $\overline{u}_j = \overline{u}_{\eta, j}, \forall \eta, j $.  By Corollary \ref{cor:mainh}, 
\begin{align}\label{eqn:eta}
\alpha_j \otimes F_{\eta}(\sqrt[p]{f}, \sqrt[p]{u_1}, \cdots, \sqrt[p]{u_{n+1}}) =0, \forall j
\end{align}
By \cite[Proposition 5.8]{hhk15} and \cite[Proposition 1.17]{boi}, there exists an irreducible open set $U_{\eta}$ of $\mathcal{X}_0$ containing $\eta$ such that 
\begin{align}\label{eqn:ueta}
\alpha_j \otimes F_{U_{\eta}}(\sqrt[p]{f}, \sqrt[p]{u_1}, \cdots, \sqrt[p]{u_{n+1}}) =0, \forall j
\end{align}
Let  $\mathscr{P}$ denote the collection  of all closed points in $\mathcal{X}$ that are not in $U_{\eta}$ for any $\eta$. Let $A$ be the semi-local ring at $\mathscr{P}$ and for $P\in \mathscr{P}$, let $A_P$ denote the local ring at $P$. By our choice of $\mathcal{X}$, each $\alpha_i$ is ramified at most at $\pi_P, \delta_P$ where $\mathfrak{m}_P = (\pi_P, \delta_P)$ is the maximal ideal at $P$. Note that the $p$-rank of the residue field  at $P$ is $n$.  By the Chinese remainder theorem, there exists $v_1, \cdots, v_n \in A$ such that  their residues  at any $P$ form a $p$-basis.  Now by \cite[Lemma 2.3]{ps_uniform},  there exists $g_1, g_2 \in F^{\times}$ such that for every $P\in\mathscr{P}$, $\mathfrak{m}_P=(g_1, g_2)$ and $g_1$  defines $C$ at all $P \in \mathscr{P}\cap C$ and $g_2$ defines $E$ at all $P\in \mathscr{P} \cap E$.  So each $\alpha_j$ is ramified at most at $g_1$ and $g_2$ at  each $P$. By Theorem \ref{thm:ndim}, we conclude that for any $P\in \mathscr{P}$,
\begin{align}\label{eqn:closedP}
\alpha_j \otimes  F_{P}(\sqrt[p]{g_1}, \sqrt[p]{g_2}, \sqrt[p]{v_1}, \cdots, \sqrt[p]{v_{n}}) =0, \forall j
\end{align}
Let  
\begin{align*}
  L:= F(\sqrt[p]{f}, \sqrt[p]{u_1}, \cdots, \sqrt[p]{u_{n+1}}, \sqrt[p]{g_1}, \sqrt[p]{g_2}, \sqrt[p]{v_1}, \cdots, \sqrt[p]{v_{n}})  
\end{align*}
Note that $[L:F]= p^{2(\mathscr{R}_{ps}(\mathcal{K})+1)}$. Consider the normal closure $\overline{\mathcal{X}}$  of $\mathcal{X}$ in $L$. Let $\phi: \mathcal{Y} \rightarrow \mathcal{X}$ be a proper birational morphism such that $\mathcal{Y}$ is regular. Let $y$ be a point in the special fiber of $\mathcal{Y}$.  To prove the first part of the theorem, it suffices to show that $\alpha_j \otimes L_y =0,  \forall j$ where $y$ is any point in the special fiber of $\mathcal{Y}$ (\cite[Theorem 9.12]{hhk_torsors}).  There are three possibilities:
\begin{enumerate}[(i)]
    \item Suppose $\phi(y) = \eta$ is a generic point of an irreducible component, then by the choice of $L$, $F_{\eta}(\sqrt[p]{f}, \sqrt[p]{u_1}, \cdots, \sqrt[p]{u_{n+1}}) \subseteq L_y$. Therefore,  by (\ref{eqn:eta})$, \alpha_j \otimes L_y =0, \forall j$.
    \item Suppose $\phi(y) = P$ is a closed point and $P\in U_{\eta}$ for some $\eta$, then $F_{U_{\eta}} \subset F_P$. Therefore, we have $F_{U_{\eta}}(\sqrt[p]{f}, \sqrt[p]{u_1}, \cdots, \sqrt[p]{u_{n+1}}) \subseteq F_{P}(\sqrt[p]{f}, \sqrt[p]{u_1}, \cdots, \sqrt[p]{u_{n+1}}) \subseteq L_y $ and by (\ref{eqn:ueta}), $\alpha_j \otimes L_y =0, \forall j$. 
    \item Suppose $\phi(y) = P$ is a closed point and $P\notin U_{\eta}$ for any $\eta$. Then we have $P \in \mathscr{P}$. Since $F_P(\sqrt[p]{g_1}, \sqrt[p]{g_2}, \sqrt[p]{v_1}, \cdots, \sqrt[p]{v_{n}}) \subseteq L_y$, $\alpha_j \otimes L_y =0$ by (\ref{eqn:closedP}).
\end{enumerate}
    This proves the first part of the theorem. For the second part of the theorem, we note that  for any finite extension $E/F$, $E(\zeta)$ is a semi-global field and the isomorphism $H^2(E(\zeta), \mu_p^{\otimes 2}) \cong  H^2(E(\zeta), \mu_p)$ yields  uniform $(2,p)$-boundedness as well as  the corresponding bound from the first part of the claim.
    \end{proof}

\section*{Acknowledgments}
The author acknowledges the support of the DAE, Government of India, under Project Identification No. RTI4001. She is grateful to David Harbater, R. Parimala and V. Suresh for the many useful discussions and for their feedback on the manuscript.

\nocite*{}
\bibliographystyle{alpha}
\bibliography{ref}

\end{document}